\documentclass[11pt,reqno]{amsart}
\usepackage{amsmath,amssymb,amsthm}
\usepackage{enumitem}
\usepackage{pgfplots}
\usepgfplotslibrary{fillbetween}
\pgfplotsset{compat=1.18}
\usepackage{hyperref}
\usepackage{doi}

\newtheorem{theorem}{Theorem}
\newtheorem{lemma}[theorem]{Lemma}
\newtheorem{corollary}[theorem]{Corollary}

\theoremstyle{definition}
\newtheorem{example}[theorem]{Example}

\theoremstyle{remark}
\newtheorem{remark}[theorem]{Remark}

\newcommand{\iterPhi}[1]{\Phi^{\circ #1}}

\numberwithin{equation}{section}

\renewcommand{\leq}{\leqslant}
\renewcommand{\geq}{\geqslant}

\begin{document}

\title[Eventually greedy best Egyptian underapproximations]{Eventually greedy best Egyptian underapproximations of rational numbers via optimal control}

\author[V.~Kova\v{c}]{Vjekoslav Kova\v{c}}
\address{University of Zagreb Faculty of Science, Department of Mathematics, Bijeni\v{c}ka cesta 30, 10000 Zagreb, Croatia}
\email{vjekovac@math.hr}

\author[Q.~Tang]{Quanyu Tang}
\address{School of Mathematical Sciences, University of Science and Technology of China, Hefei 230026, P. R. China}
\email{tangquanyu827@gmail.com}

%Number theory - Rational numbers as sums of fractions
%Calculus of variations and optimal control; optimization - Applications of optimal control and differential games

\begin{abstract}
We prove that every positive rational number has eventually greedy best Egyptian underapproximations, both when repetitions of the denominators are allowed and when the denominators are required to be distinct. This answers affirmatively a problem originating with Erd\H{o}s and Graham and later revisited by Nathanson, and yields an application concerning the maximal asymptotic growth of denominators in unit fraction series converging to a given rational number. We reformulate the question as an optimal control problem for a dynamical system, construct an appropriate payoff function, and study properties of the associated Bellman function. We also answer another question of Nathanson by constructing an irrational number with unique and greedy best Egyptian underapproximations.
\end{abstract}

\maketitle

%\tableofcontents

\section{Introduction}
A finite sum of distinct unit fractions with integer denominators,
\[
\frac{1}{a_1}+\frac{1}{a_2}+\cdots+\frac{1}{a_n},
\quad
2\leq a_1<a_2<\cdots<a_n,
\]
is commonly called an \emph{Egyptian fraction}. An early systematic discussion of such sums appears in Fibonacci's \emph{Liber Abaci}, written in $1202$, which describes a finite greedy procedure for decomposing every positive proper rational number into distinct unit fractions \cite{Fibonacci}. Nowadays, problems concerning such sums often require nontrivial tools from combinatorial number theory, Fourier analysis, or Diophantine approximation. Unit fractions were also of particular interest to Paul Erd\H{o}s; surveys from that viewpoint were written by Schinzel \cite{Schinzel}, Graham \cite{Graham}, and Bloom and Elsholtz \cite{BloomElsholtz}. Since there has been some even more recent progress on the topic, we advise the reader to track the ongoing developments on Bloom's website \emph{Erd\H{o}s Problems} \cite{Bloom}.

One line of research concerns best and greedy underapproximations. An \emph{Egyptian underapproximation} of a real number $\lambda>0$ is a finite sum of distinct unit fractions that is strictly smaller than $\lambda$. The largest such admissible sum with $n$ terms is called the \emph{best $n$-term Egyptian underapproximation} of $\lambda$. The number $\lambda=1$ has the property that its best underapproximations are formed by successively taking the denominators from the Sylvester sequence \cite{Sylvester},
\[
a_1=2,
\quad
a_{n+1}=a_n^2-a_n+1 \quad\text{for } n\geq1.
\]
The corresponding best $n$-term underapproximations for $n=1,2,3,\ldots$ are
{\allowdisplaybreaks\begin{align*}
\frac12&<1,\\
\frac12+\frac13&=\frac56<1,\\
\frac12+\frac13+\frac17&=\frac{41}{42}<1,\\
\frac12+\frac13+\frac17+\frac1{43}&=\frac{1805}{1806}<1,\ldots.
\end{align*}}
This property was conjectured by Miller \cite{Miller} and Kellogg \cite{Kellogg} and proved independently by Curtiss \cite{Curtiss} and Takenouchi \cite{Takenouchi}; also see Soundararajan \cite{Soundararajan}.
Starting from a different positive real number $\lambda$, one can instead construct Egyptian underapproximations recursively. Given $\sum_{i=1}^n1/a_i$, the \emph{greedy} choice is the unit fraction $1/a_{n+1}$ with the smallest denominator $a_{n+1}$ larger than $a_n$ for which one still has $\sum_{i=1}^{n+1}1/a_i<\lambda$. It is then natural to ask whether these greedy underapproximations coincide with the best ones.
Erd\H{o}s observed that this phenomenon, in addition to $1$, also holds for every unit fraction \cite{Erdos}, while Nathanson \cite{Nathanson} and Chu \cite{Chu} extended it further to certain other rationals.
For a general rational number, however, the greedy choice need not be optimal. For example, the greedy two-term underapproximation of $10/61$ is
\[
\frac17+\frac1{48},
\]
whereas
\[
\frac19+\frac1{19} \in \Bigl( \frac17+\frac1{48}, \frac{10}{61} \Bigr)
\]
is a better underapproximation of the same number.

It is therefore natural to ask whether all positive rational numbers $\lambda$ have \emph{eventually greedy} best Egyptian underapproximations. In other words, even if the greedy algorithm fails to give the best underapproximation for small values of $n$, it is plausible that, for all sufficiently large $n$, a best $(n+1)$-term underapproximation is obtained from a best $n$-term underapproximation $\sum_{i=1}^n1/a_i$ by appending one further unit fraction $1/a_{n+1}$ chosen greedily. Erd\H{o}s and Graham stated the corresponding assertion in \cite[p.~31]{ErdosGraham}, but gave neither a proof nor a reference. Graham later treated it explicitly as an open question \cite[pp.~295--296]{Graham}:
\begin{quote}
(\ldots) \emph{is it true that for any rational $\frac{a}{b}$, the closest strict under approximation $R_n(\frac{a}{b})$ of $\frac{a}{b}$ by a sum of $n$ unit fractions is given by
\[
R_n\Bigl(\frac{a}{b}\Bigr) = R_{n-1}\Bigl(\frac{a}{b}\Bigr) + \frac{1}{m}
\]
where $m$ is the least denominator not yet used for which $R_n(\frac{a}{b})<\frac{a}{b}$, provided that $n$ is sufficiently large?}
\end{quote}
The same open problem was later posed by Nathanson \cite[Open problem~(4)]{Nathanson}, one of the present authors \cite[p.~42]{Kovac}, and Li and one of the present authors \cite[Conjecture~1.5]{LiTang}. It is also listed as a remaining open question in the comments accompanying \cite[Problem~\#206]{Bloom}. In the opposite direction, one of the present authors showed that the set of real numbers with the eventual greediness property has Lebesgue measure zero \cite[Theorem~1]{Kovac}. Since the rational numbers form a null set, that result remains compatible with the possibility that every rational number has the property. Finally, the papers \cite{Nathanson} and \cite{LiTang} both work with a convention that permits equal denominators. In the latter work, the above conjecture was used as a hypothesis in a study of asymptotic extremality for Sylvester-type sequences. Since much of the literature on Egyptian fractions instead requires the denominators to be distinct, we study both conventions: we consider separately nondecreasing and strictly increasing denominator tuples, and continue to use the term ``Egyptian underapproximation'' under both conventions.

Throughout, $\mathbb{N}=\{1,2,\ldots\}$, and all logarithms are natural. For $n\geq1$, set
\begin{align*}
\mathcal{E}_n^{\leq}
&:=\bigl\{(a_1,\ldots,a_n)\in\mathbb{N}^n:
2\leq a_1\leq\cdots\leq a_n\bigr\},\\
\mathcal{E}_n^{<}
&:=\bigl\{(a_1,\ldots,a_n)\in\mathbb{N}^n:
2\leq a_1<\cdots<a_n\bigr\}.
\end{align*}
The superscripts $\leq$ and $<$ will remind us, respectively, of nondecreasing and strictly increasing denominators.
For $\lambda>0$ and $\sigma=\leq$ or $\sigma=<$, define
\[
R_n^{\sigma}(\lambda)
=\max\left\{
\sum_{i=1}^n\frac1{a_i} \,:\,
(a_1,\ldots,a_n)\in\mathcal{E}_n^{\sigma},\ \
\sum_{i=1}^n\frac1{a_i}<\lambda
\right\}.
\]
It is easy to justify that this maximum exists; see \cite[Theorem~3]{Nathanson}.
A tuple in $\mathcal{E}_n^\sigma$ attaining it will be called a \emph{maximizing $n$-term denominator tuple}.
The quantity $R_n^{<}(\lambda)$ is what Graham simply wrote as $R_n(\lambda)$.

For a number $u>0$, the smallest integer $x\geq2$ for which $1/x<u$ is
\begin{equation}\label{eq:def-of-G}
G(u) := \max\left\{ \Bigl\lfloor\frac{1}{u}\Bigr\rfloor+1,\, 2 \right\}.
\end{equation}
Thus, $1/G(\lambda)$ is trivially the best one-term Egyptian underapproximation of $\lambda>0$.
Two-term best underapproximations have been studied by Nathanson \cite{Nathanson} and Chu \cite{Chu}.
Nathanson's formulation \cite[Open problem~(4)]{Nathanson} of the Erd\H{o}s--Graham problem asserts that for every rational $\lambda\in(0,1]$ there is an integer $n_0\geq1$ such that, for every $m\geq1$,
\[
R_{n_0+m}^{\leq}(\lambda)
=R_{n_0}^{\leq}(\lambda)
+R_m^{\leq}\bigl(
\lambda-R_{n_0}^{\leq}(\lambda)
\bigr),
\]
and that the second term on the right is obtained by the greedy underapproximation algorithm.

Our main result proves this conjecture and its strictly increasing counterpart for all positive rational numbers.

\begin{theorem}
\label{thm:main}
Let $\sigma$ be either $\leq$ or $<$.
For every positive rational number $\lambda$, there is an integer $n_0=n_0(\lambda,\sigma)\geq1$ such that, for every $m\geq1$,
\begin{equation}
R_{n_0+m}^{\sigma}(\lambda)
=R_{n_0}^{\sigma}(\lambda)
+R_m^{\sigma}\bigl(
\lambda-R_{n_0}^{\sigma}(\lambda)
\bigr),
\label{eq:main-decomposition}
\end{equation}
and the greedy $m$-term denominator tuple is the unique maximizing tuple for the remainder $\lambda-R_{n_0}^{\sigma}(\lambda)$. Moreover, one can choose a maximizing denominator tuple for $R_{n_0}^{\sigma}(\lambda)$ such that, for every $m\geq1$, adjoining the first $m$ greedy denominators produces a maximizing tuple for $R_{n_0+m}^{\sigma}(\lambda)$.
\end{theorem}

Most of the paper is dedicated to the proof of Theorem~\ref{thm:main}.
To interest the reader, we mention that the techniques will be quite unusual for this context, as we will reformulate the theorem as a dynamical optimization problem, and thus use the methods and language of the theory of optimal control.
However, no prior knowledge of the control of dynamical systems will be needed, and all required definitions, concepts, and ideas will be given in Section~\ref{sec:scheme}.

Erd\H{o}s and Graham in fact suspected that all positive algebraic numbers could possess the same property \cite[p.~31]{ErdosGraham}; we are not in a position to resolve this variant of the problem here. Furthermore, Nathanson \cite[Open problem~(1)]{Nathanson} asked whether there exist irrational numbers whose greedy underapproximations are uniquely best for any fixed number of terms. The following example gives an affirmative answer.

\begin{example}\label{ex:Liouville}
The number
\[
\theta := \sum_{n=1}^{\infty} \frac{1}{2^{n!}}
= \frac{1}{2} + \frac{1}{4} + \frac{1}{64} + \cdots
\]
is Liouville (and thus irrational) and has a unique best $n$-term Egyptian underapproximation for every $n\geq1$, namely its $n$-term greedy underapproximation, in both the nondecreasing and strictly increasing denominator conventions.
\end{example}

The claim from Example~\ref{ex:Liouville} is proved in Section~\ref{sec:example}; parts of the argument resemble the proof of Theorem~3.5 in \cite{Chu2}. This explicit irrational example is noteworthy because, as recalled above, almost every real number fails to have the eventual greediness property \cite[Theorem~1]{Kovac}.

The proof of Theorem~\ref{thm:main} in the case of nondecreasing denominators starts with a positive rational number $\lambda$ and constructs an integer $N\geq1$, a maximizing tuple $(b_1,\ldots,b_N)$, and an integer $T_0\geq2$ such that
\[
\lambda=\sum_{i=1}^N\frac1{b_i}+\frac1{T_0}
\]
and the following property holds.
If we define recursively
\begin{equation}
T_{j+1}=T_j(T_j+1)\quad\text{for } j\geq0,
\label{eq:unit-tail-recurrence}
\end{equation}
and then extend the terms $b_1,b_2,\ldots,b_N$ into the full infinite sequence $(b_n)_{n\geq1}$ by
\[
b_{N+j+1}=T_j+1\quad\text{for } j\geq0,
\]
then
\[
\sum_{n=1}^{\infty}\frac1{b_n}=\lambda,
\quad
\sum_{i=1}^n\frac1{b_i}=R_n^{\leq}(\lambda)
\]
for every $n\geq N$.
Moreover, the constructed sequence $(b_n)_{n\geq1}$ grows double exponentially, the limit $\lim_{n\to\infty}b_n^{2^{-n}}$ exists, and it is finite and strictly positive; see the classical analysis of recurrent sequences by Aho and Sloane \cite{AhoSloane}.

\begin{corollary}
\label{cor:li-tang}
Let $\lambda>0$ be rational and suppose that, for every positive integer $n$, the maximizing $n$-term denominator tuple for $R_n^{\leq}(\lambda)$ is unique. Let $(b_n)_{n\geq1}$ be the sequence defined above, and let $(a_n)_{n\geq1}$ be a sequence of integers satisfying
\begin{equation}\label{eq:equalslambda}
2\leq a_1\leq a_2\leq a_3\leq\cdots,
\quad
\sum_{n=1}^{\infty}\frac1{a_n}=\lambda.
\end{equation}
If $(a_n)_{n\geq1}\neq(b_n)_{n\geq1}$, then
\[
\liminf_{n\to\infty}a_n^{2^{-n}}
<
\lim_{n\to\infty}b_n^{2^{-n}}.
\]

\end{corollary}

This result was posed as \cite[Conjecture~4.1]{LiTang} by Li and one of the present authors, who also showed \cite[Theorem~1.6]{LiTang} that it would follow from the conjecture that all rational numbers have eventually greedy best Egyptian underapproximations, stated in the same paper as \cite[Conjecture~1.5]{LiTang}.
Since the latter ingredient is provided by our Theorem~\ref{thm:main}, Corollary~\ref{cor:li-tang} now becomes a fully established result.
A minor distinction is that the conditional result in \cite{LiTang} was formulated with a restriction to $\lambda\leq1$ and for strictly increasing $(a_n)_{n\geq1}$. These additional assumptions are not important for the proof. Anyway, they are addressed below in the more general Theorem~\ref{thm:asym}. 

Known $n$-term uniqueness theorems turn Corollary~\ref{cor:li-tang} into concrete results without the abstract uniqueness hypothesis.
Let $\lambda=p/q\in(0,1]$ be a reduced fraction and suppose that either
\begin{enumerate}[label=\textup{(\roman*)}]
\item $p\mid q+1$, or
\item $q$ is odd and $2$ is the least positive integer $\ell$ for which $p\mid q+\ell$.
\end{enumerate}
Let $(b_n)_{n\geq1}$ be the sequence obtained by applying the greedy underapproximation algorithm indefinitely to $p/q$. Then every sequence of integers $2\leq a_1<a_2<\cdots$ such that
\[
\sum_{n=1}^{\infty}\frac{1}{a_n}=\frac{p}{q}, \quad (a_n)_{n\geq1}\neq(b_n)_{n\geq1}
\]
satisfies
\[
\liminf_{n\to\infty}a_n^{2^{-n}}
<
\lim_{n\to\infty}b_n^{2^{-n}}.
\]
Namely, when $p=q=1$, the Curtiss--Takenouchi theorem gives the required uniqueness at every length \cite{Curtiss,Takenouchi}. When $p<q$, Nathanson proves it under the first condition \cite[Theorem~5]{Nathanson}, and Chu proves it under the second \cite[Theorem~1.12]{Chu}. In each case the unique maximizing $n$-tuple consists of the first $n$ terms of $(b_n)_{n\geq1}$, so Corollary~\ref{cor:li-tang} applies. For $\lambda=1$, this specializes to a question by Erd\H{o}s and Graham, proved independently by Li and one of the present authors \cite{LiTangSylvester} and by Kamio \cite{Kamio}; also see \cite[Problem~\#315]{Bloom}. Moreover, if $\lambda=1$, the number
\begin{equation}\label{eq:Vardi-com}
\lim_{n\to\infty}b_n^{2^{-n}} = 1.2640847353\ldots
\end{equation}
is known as the \emph{Vardi constant}; see \cite[Exercise~4.37, p.~518]{GKP} or \cite{Vardi}.

The possibility of the most general result along the lines of \cite[Problem~\#315]{Bloom} has been suggested to us by Wouter van Doorn. Namely, the uniqueness hypothesis in Corollary~\ref{cor:li-tang} is unnecessary if the conclusion is replaced by a dichotomy between two possible types of behavior of $(a_n)_{n\geq1}$.

\begin{theorem}\label{thm:asym}
Let $\lambda>0$ be a rational number and let $(b_n)_{n\geq1}$ be the sequence constructed before the formulation of Corollary~\ref{cor:li-tang}. Then every sequence of integers $(a_n)_{n\geq1}$ satisfying \eqref{eq:equalslambda} has one of the following two properties:
\[
\liminf_{n\to\infty}a_n^{2^{-n}}
<
\lim_{n\to\infty}b_n^{2^{-n}}
\]
or
\[
a_n=b_n\quad\text{for all sufficiently large }n.
\]
\end{theorem}

In particular, Theorem~\ref{thm:asym} identifies the maximal possible double exponential growth rate of denominators in a unit fraction series converging to a positive rational number $\lambda$.
Since our original intention in formulating Corollary~\ref{cor:li-tang} was to avoid repeating the arguments from \cite{LiTang}, we somewhat briefly describe the main ideas and necessary modifications for the proof of Theorem~\ref{thm:asym} in Section~\ref{sec:asymptotic-rigidity}.

%%%%%

\section{Scheme of the proof of Theorem~\ref{thm:main}}
\label{sec:scheme}

The proof uses ideas from the optimal control of dynamical systems and it is primarily inspired by the Bellman function method. Bellman functions originate in dynamic programming \cite{Bellman}, their use in sharp martingale inequalities goes back to Burkholder \cite{Burkholder}, and their role in harmonic analysis is surveyed in \cite{NazarovTreilVolberg,VasyuninVolberg}. To the best of our knowledge, this is the first application of optimal control techniques to an extremal problem for Egyptian fractions. Somewhat related is Eppstein's use of the shortest path dynamic programming to study Egyptian fractions \cite{Eppstein}, while dynamical systems have already been applied to expansions into infinite series of unit fractions \cite{Galambos,Robinson}.
The reader should rest assured that no prior knowledge of optimal control is needed for full comprehension of the proof.

%%%%%

\subsection{Dynamical reformulation}
In this section, for simplicity, we only talk about decompositions of rational numbers $\lambda\in(0,1]$ into sums of distinct unit fractions. The next section will be entirely general and establish Theorem~\ref{thm:main} in the form stated in the introduction.
The problem of computing $R_n^{<}(\lambda)$ for fixed $\lambda$ and $n$ is a static optimization problem. The following proof converts it into a dynamical optimization problem: we consider all possible partial decompositions of $\lambda=p/q$ as
\[
\frac{p}{q}
=\sum_{i=1}^{m}\frac{1}{x_i}+\frac{P}{Q},
\quad
2\leq x_1<\cdots<x_m,
\quad x_1,\ldots,x_m\in\mathbb{N},
\]
where $P/Q$ is a positive fraction, interpreted as a \emph{remainder}. Of course, the remainder can be decomposed further, splitting the original problem into subproblems in the following precise way.

Write the current remainder as $P/Q$, not necessarily in lowest terms, i.e., $P$ and $Q$ need not be coprime. Let $L$ be the largest denominator already selected, i.e., $L=x_m$ in the above representation. We regard the triple
\[ s=(P,Q,L) \]
as the \emph{state} of a controlled dynamical system. A denominator $x$ is an admissible \emph{control} when
\[
x>\max\Bigl\{\frac{Q}{P},L\Bigr\}.
\]
Since
\[
\frac{P}{Q}-\frac{1}{x}
=\frac{Px-Q}{Qx},
\]
it produces the \emph{successor} state
\begin{equation}\label{eq:successor-state}
s[x]:=(Px-Q,Qx,x),
\end{equation}
describing the new remainder and the newly added denominator.
Thus \emph{trajectories} starting at $(p,q,0)$, where $\lambda=p/q$, are precisely strictly increasing Egyptian underapproximations of $\lambda$. In particular, maximizing the underapproximation after $n$ controls is equivalent to minimizing the remainder after $n$ controls.
However, there is some freedom in ways of measuring the size of the remainder, and we will need to carefully choose an appropriate cost/payoff function below.

We next consider \emph{terminal decompositions}
\begin{equation}\label{eq:ter-dec}
\frac{P}{Q}
=\sum_{i=1}^{m}\frac{1}{x_i}+\frac{1}{T},
\quad
L<x_1<\cdots<x_m<T,
\end{equation}
with the evident interpretation when $m=0$.
After choosing the \emph{payoff function} $\mathcal{H}\colon[1,\infty)\to[0,\infty)$, we assign to each terminal decomposition \eqref{eq:ter-dec} the \emph{value} $2^{-m}\mathcal{H}(T)$.
Note that this value is obtained by ``discounting'' the payoff by the factor $2^{-m}$, which depends on the number of chosen denominators.
The \emph{Bellman function}
\[ \mathcal{B}(s) = \mathcal{B}(P,Q,L) \]
is then defined as the supremum of these values over all terminal decompositions for fixed values of $P$, $Q$, and $L$. Thus, our new dynamical optimization problem consists of estimating $\mathcal{B}(p,q,0)$, i.e., maximizing the ``terminal completion value'' of the initial state $(p,q,0)$.
Separating terminal decompositions according to whether the system stops immediately or takes a first control $x$ gives the dynamic programming identity
\[
\mathcal{B}(s)
=
\max\biggl\{
\underbrace{\mathcal{H}(T)}_{\substack{\text{if $P/Q=1/T$}\\ \text{and $L<T$}}},
\frac{1}{2}\sup_{x}\mathcal{B}(s[x])
\biggr\};
\]
see Lemma~\ref{lem:bellman-principle} for a bit more precise version.
Namely, the classical extremal theorem for unit fractions will supply the \emph{boundary condition} $\mathcal{B}(P,Q,L)=\mathcal{H}(T)$ whenever $P/Q$ is already a unit fraction $1/T$ and $T>L$.

A simpler dynamics appears once the remainder is a unit fraction. Indeed, the strict greedy choice from $1/T$ is $1/(T+1)$, and the new remainder is
\[
\frac{1}{T}-\frac{1}{T+1}
=\frac{1}{T(T+1)}.
\]
Consequently, the reciprocal of the remainder evolves under the map
\begin{equation}\label{eq:def-of-Phi}
\Phi(t)=t(t+1)\quad\text{for } t\geq1.
\end{equation}
Write
\[
\iterPhi{j} := \underbrace{\Phi \circ \Phi \circ \cdots \circ \Phi}_{j\text{ copies of }\Phi},
\]
with the convention that $\iterPhi{0}$ is the identity.

Lemma~\ref{lem:payoff-construction} below will then construct the payoff function $\mathcal{H}$ as a unique strictly increasing function such that
\begin{equation}\label{eq:H-log-bounds}
\mathcal{H}(\Phi(t))=2\mathcal{H}(t),
\quad
\log t\leq\mathcal{H}(t)\leq\log t+\log 2.
\end{equation}
Hence, the value $2^{-m}\mathcal{H}(T)$ of \eqref{eq:ter-dec} remains the same when a greedy term is appended to the unit fraction remainder. Lemma~\ref{lem:greedy-bellman-drift} shows, moreover, that this value strictly increases along trajectories consisting of states appearing in more general greedy trajectories leading to a unit fraction.

A substantive arithmetic ingredient is the proof that $\mathcal{B}(P,Q,L)$ is finite and attained. After applying controls $x_1,\ldots,x_j$, define common denominators $Q_j$ and ``cleared'' numerators $C_j$ by
\begin{equation}
Q_j:=Q\prod_{i=1}^j x_i,
\quad
C_j:=Q_j\left(
\frac{P}{Q}-\sum_{i=1}^j\frac1{x_i}
\right),
\quad C_0:=P.
\label{eq:cleared-numerators}
\end{equation}
(Once again, $C_j/Q_j$ need not be in lowest terms.)
A single control need not decrease $C_j$. We therefore stop at the first index $r$ for which
\[
C_j\geq P\quad\text{for } j<r,
\quad
C_r=d<P.
\]
The comparison argument in Lemma~\ref{lem:bellman-transition} shows that a branch with no such descent cannot improve upon the ``reference'' payoff $\mathcal{H}(Q/P)$. At the first descent, it gives
\[
t_r<\frac{P}{d}\,\Phi^{\circ r}(t),
\quad
x_r<t_r,
\quad
t=\frac{Q}{P},
\]
where $t_r$ is the reciprocal of the new remainder. The second inequality makes the successor state admissible, while $d<P$ allows strong induction on the cleared numerator. This induction will show that $\mathcal{B}(P,Q,L)$ is finite and attained; see Lemmas~\ref{lem:bellman-majorant} and~\ref{lem:bellman-attainment}.

For a rational number $\lambda=p/q$ that is not a unit fraction, the aforementioned attainment of $\mathcal{B}(p,q,0)$ guarantees an optimal terminal decomposition
\[
\lambda
=\sum_{i=1}^{N}\frac{1}{a_i}+\frac{1}{T},
\quad
a_1<\cdots<a_N<T,
\]
such that
\[
\Lambda
:=2^{-N}\mathcal{H}(T)
=\mathcal{B}(p,q,0).
\]
Put $T_0=T$ and define recursively $T_{j+1}=T_j(T_j+1)$ as in the introduction. Appending the greedy denominators $T_0+1,\ldots,T_{k-1}+1$ gives an $(N+k)$-term underapproximation with error $1/T_k$. Using the above properties of the simpler ``unit fraction dynamics'' shows that the values of the extended terminal decompositions stay the same,
\[
2^{-(N+k)}\mathcal{H}(T_k)=\Lambda.
\]
Notice that this proves Bellman optimality of the resulting terminal decompositions, but does not by itself prove their optimality among all underapproximations with a fixed number of terms.

That final implication is supplied by Lemma~\ref{lem:completion}. Given a sufficiently good $n$-term underapproximation with error $h$, greedily complete $h$ to a unit fraction. Whenever a new greedy denominator collides with an original denominator, delete the colliding original term and restart.
After $r\leq n$ deletions, this produces a modified error $h'$ satisfying
\[
h'\leq2^r h
\]
and a terminal decomposition for which
\[
\mathcal{B}(p,q,0)
\geq
2^{-(n-r)}\mathcal{H}(1/h').
\]

Suppose now that $n=N+k$ and that an $n$-term underapproximation is better than the above greedy extension. Its error satisfies $h<1/T_k$, and the completion lemma (i.e., Lemma~\ref{lem:completion}) and strict monotonicity of $\mathcal{H}$ give
\[
\Lambda
>
2^{-n}2^r
\mathcal{H}\Bigl(\frac{T_k}{2^r}\Bigr).
\]
For all sufficiently large $k$, the logarithmic bounds for $\mathcal{H}$ from \eqref{eq:H-log-bounds} and the double exponential growth of $T_k$ imply, uniformly in $0\leq r\leq n$, that
\[
2^r\mathcal{H}\Bigl(\frac{T_k}{2^r}\Bigr)
\geq\mathcal{H}(T_k).
\]
The preceding displayed strict inequality therefore contradicts $\Lambda=2^{-n}\mathcal{H}(T_k)$.
Thus, all sufficiently long greedy extensions of the optimal terminal decomposition are best underapproximations.
This will prove \eqref{eq:main-decomposition}, the asserted uniqueness of the maximizing denominator tuple for the residual problem, and the existence of the compatible nested maximizing sequence of denominator tuples.

%%%%%

\subsection{Construction of the payoff function}
Let $\Phi$ be defined as in \eqref{eq:def-of-Phi}. 
As we have seen in the previous subsection, appending the next greedy term to a terminal decomposition with remainder $1/T$ changes the remainder to $1/\Phi(T)$. This motivates us to introduce the following payoff function.

\begin{lemma}
\label{lem:payoff-construction}
There exists a unique function
\[
\mathcal{H}\colon[1,\infty)\longrightarrow[0,\infty)
\]
with the following properties:
\begin{enumerate}[
 label=\textup{(H\arabic*)},
 ref=\textup{(H\arabic*)}
]
\item\label{item:payoff-H1}
$\mathcal{H}(t(t+1))=2\mathcal{H}(t)$ for $t\geq1$;
\item\label{item:payoff-H2}
$\mathcal{H}$ is strictly increasing on $[1,\infty)$;
\item\label{item:payoff-H3}
$\log t\leq \mathcal{H}(t)\leq\log t+\log 2$ for $t\geq1$.
\end{enumerate}
\end{lemma}

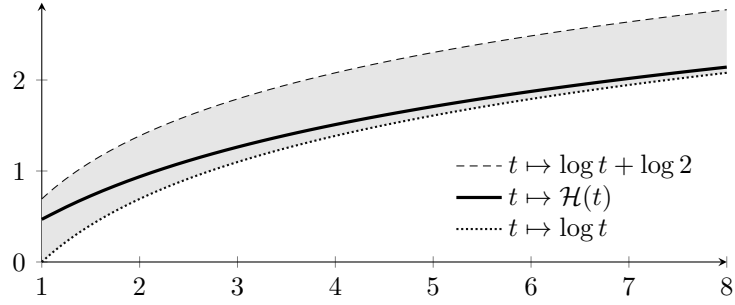
\begin{figure}
\begin{center}
\begin{tikzpicture}
\begin{axis}[
  width=.68\textwidth,
  height=.32\textwidth,
  xmin=1, xmax=8,
  ymin=0, ymax=2.85,
  axis lines=left,
  xtick={1,2,3,4,5,6,7,8},
  ytick={0,1,2},
  tick label style={font=\small},
  label style={font=\small},
  legend style={
    at={(.97,.04)},
    anchor=south east,
    draw=none,
    fill=none,
    font=\small,
    row sep=-1pt
  },
  legend cell align=left,
  clip=false
]
\addplot[forget plot, name path=lowerpath, draw=none,
  domain=1:8, samples=120] {ln(x)};
\addplot[forget plot, name path=upperpath, draw=none,
  domain=1:8, samples=120] {ln(2*x)};
\addplot[forget plot, black!10]
  fill between[of=lowerpath and upperpath];

\addplot[black!95, densely dashed, domain=1:8, samples=120]
  {ln(2*x)};
\addlegendentry{$t\mapsto\log t+\log 2$}

\addplot[very thick, black, smooth] coordinates {
  (1.00,0.46869666)
  (1.25,0.60473251)
  (1.50,0.72698926)
  (1.75,0.83727461)
  (2.00,0.93739332)
  (2.25,1.02889274)
  (2.50,1.11304541)
  (2.75,1.19088774)
  (3.00,1.26326528)
  (3.25,1.33087230)
  (3.50,1.39428348)
  (3.75,1.45397853)
  (4.00,1.51036130)
  (4.25,1.56377450)
  (4.50,1.61451117)
  (4.75,1.66282362)
  (5.00,1.70893056)
  (5.25,1.75302269)
  (5.50,1.79526727)
  (5.75,1.83581172)
  (6.00,1.87478665)
  (6.25,1.91230827)
  (6.50,1.94848046)
  (6.75,1.98339636)
  (7.00,2.01713990)
  (7.25,2.04978687)
  (7.50,2.08140600)
  (7.75,2.11205978)
  (8.00,2.14180517)
};
\addlegendentry{$t\mapsto\mathcal{H}(t)$}

\addplot[thick, black!95, densely dotted, domain=1:8, samples=120]
  {ln(x)};
\addlegendentry{$t\mapsto\log t$}
\end{axis}
\end{tikzpicture}
\end{center}
\caption{Graph of $\mathcal{H}$.}
\label{fig:graph-of-H}
\end{figure}

\begin{proof}
For existence, fix $t\geq1$, put $t_0=t$, and define recursively $t_{j+1}=\Phi(t_j)$ for $j\geq0$. In other words, $t_j:=\iterPhi{j}(t)$. Iterating
\[
\log t_{j+1}
=2\log t_j+\log\Bigl(1+\frac1{t_j}\Bigr)
\]
gives
\[
2^{-j}\log t_j
=\log t+
\sum_{\ell=0}^{j-1}2^{-\ell-1}
\log\Bigl(1+\frac1{t_\ell}\Bigr).
\]
The series converges uniformly on $[1,\infty)$ because its $\ell$th summand is at most $2^{-\ell-1}\log 2$. Consequently, the limit
\begin{equation}\label{eq:Hformula}
\mathcal{H}(t):=\lim_{j\to\infty}2^{-j}
\log\bigl(\iterPhi{j}(t)\bigr)
\end{equation}
exists and is given by
\[
\mathcal{H}(t)=\log t+
\sum_{\ell=0}^{\infty}2^{-\ell-1}
\log\Bigl(1+\frac1{t_\ell}\Bigr).
\]
This proves existence and \ref{item:payoff-H3}. The definition immediately yields
\[
\mathcal{H}(\Phi(t))=2\mathcal{H}(t),
\]
which verifies \ref{item:payoff-H1}.

It remains to prove \ref{item:payoff-H2}. Fix $1\leq s<t$ and write
\[
s_j:=\iterPhi{j}(s),
\quad
t_j:=\iterPhi{j}(t),
\quad
\rho_j:=\frac{t_j}{s_j}.
\]
Then, since $\Phi$ is obviously strictly increasing, $\rho_j>1$ and
\[
\frac{\rho_{j+1}}{\rho_j}
=\frac{t_j+1}{s_j+1}>1.
\]
Thus, $(\rho_j)_{j\geq0}$ is increasing. It is unbounded, for if it converged to a finite limit $\rho>1$, then $s_j\to\infty$ and
\[
\frac{\rho_{j+1}}{\rho_j}
=\frac{\rho_js_j+1}{s_j+1}
\longrightarrow\rho>1,
\]
whereas convergence of $(\rho_j)_{j\geq0}$ would force the same ratio to tend to $1$. Choose $j$ with $t_j>2s_j$. By \ref{item:payoff-H1} and \ref{item:payoff-H3},
\[
\mathcal{H}(t)-\mathcal{H}(s)
=2^{-j}\bigl(\mathcal{H}(t_j)-\mathcal{H}(s_j)\bigr)
\geq2^{-j}\log\frac{t_j}{2s_j}>0.
\]
Hence $\mathcal{H}$ is strictly increasing.

Finally, suppose that $\widetilde{\mathcal{H}}$ is another function satisfying \ref{item:payoff-H1}--\ref{item:payoff-H3}. Iterating \ref{item:payoff-H1} and then applying \ref{item:payoff-H3}, we obtain, for every $j\geq0$,
\[
0\leq
\widetilde{\mathcal{H}}(t)
-2^{-j}\log\bigl(\iterPhi{j}(t)\bigr)
=2^{-j}\left(
\widetilde{\mathcal{H}}\bigl(\iterPhi{j}(t)\bigr)
-\log\bigl(\iterPhi{j}(t)\bigr)
\right)
\leq2^{-j}\log 2.
\]
Letting $j\to\infty$ and using formula \eqref{eq:Hformula} gives $\widetilde{\mathcal{H}}(t)=\mathcal{H}(t)$ for every $t\geq1$. This proves uniqueness.
\end{proof}

Note that the above proof actually provides a formula for $\mathcal{H}(t)$; it is given by \eqref{eq:Hformula}. The following interpretation of this formula for integer values of $t$ has been communicated to us by Xiao Hu.

\begin{remark}\label{rem:Hinterpretation}
The payoff function has a particularly transparent arithmetic interpretation at positive integer arguments. Fix $t\in\mathbb{N}$, put $T_0:=t$, and define $(T_j)_{j\geq0}$ recursively by \eqref{eq:unit-tail-recurrence}.
By Lemma~\ref{lem:unit-fraction-best} below, the unique best $n$-term underapproximation of $1/t$ is obtained from the beginning of this Sylvester-type iteration, i.e.,
\[
R_n^{\leq}\Bigl(\frac1t\Bigr)
=\sum_{j=0}^{n-1}\frac1{T_j+1}
=\frac1t-\frac1{T_n}.
\]
Consequently,
\[
\mathcal{H}(t)
=\lim_{n\to\infty}2^{-n}\log T_n
=\log\left(\lim_{n\to\infty}T_n^{2^{-n}}\right).
\]
Thus, at integer arguments, $\mathcal{H}(t)$ is precisely the logarithm of the asymptotic growth rate of the reciprocal errors along the optimal Sylvester-type trajectory.
By setting $b_n:=T_{n-1}+1$ to be the actual denominators and in the notation of Corollary~\ref{cor:li-tang},
\[
\mathcal{H}(t)
=2\log\left(\lim_{n\to\infty}b_n^{2^{-n}}\right).
\]
In particular, $e^{\mathcal{H}(1)/2}$ is equal to the Vardi constant \eqref{eq:Vardi-com}.

For a non-integer rational number $t\geq1$, the limit formula \eqref{eq:Hformula} remains valid, but it no longer has this direct extremal interpretation. The orbit $t,\Phi(t),\iterPhi{2}(t),\ldots$ does not describe the reciprocal errors of best underapproximations of $1/t$. Instead, Theorem~\ref{thm:main} first selects a finite maximizing prefix whose remainder is a unit fraction. Only after that prefix does the forced Sylvester-type tail begin.
\end{remark}

%%%%%

\section{Complete proof of Theorem~\ref{thm:main}}

Throughout this section, $\mathcal{H}$ always denotes the payoff function from Lemma~\ref{lem:payoff-construction}.

\subsection{Two control problems}

We first specify the control problems in more detail than before.
The state space is
\[
\mathfrak{S}=\mathbb{N}^2\times(\mathbb{N}\cup\{0\}).
\]
The state $s=(P,Q,L)$ records the current remainder $P/Q$ and the largest denominator $L$ previously selected. We call $s$ \emph{admissible}
if
\begin{equation}
LP<Q.
\label{eq:admissible-state}
\end{equation}
We define the control problem on all of $\mathfrak{S}$, not only on its admissible part, because applying a single control need not preserve \eqref{eq:admissible-state}. When $P\leq Q$, Lemma~\ref{lem:bellman-transition} below returns a trajectory to the admissible region when its cleared numerator first falls below $P$. Initial remainders larger than $1$ will be treated separately.

At $s=(P,Q,L)$ the two sets of admissible controls are
\begin{align}
\mathfrak{A}_{\leq}(s)
&=\left\{x\in\mathbb{N}:x>\frac QP,\ x\geq 2,\ x\geq L\right\},\nonumber\\
\mathfrak{A}_{<}(s)
&=\left\{x\in\mathbb{N}:x>\frac QP,\ x\geq 2,\ x>L\right\}.
\label{eq:control-set}
\end{align}
Choosing $x$ from $\mathfrak{A}_{\leq}(s)$ or $\mathfrak{A}_{<}(s)$ subtracts $1/x$ from the current remainder and produces the successor state $s[x]$ defined by \eqref{eq:successor-state}.
The condition $x>Q/P$ makes the new numerator positive, so $s[x]\in\mathfrak{S}$.

When $\sigma$ stands for either $\leq$ or $<$, a terminal decomposition from $s=(P,Q,L)$ is, as before, a choice of an integer $m\geq0$ and integers $x_1,\ldots,x_m\geq2$, $T\geq2$ such that
\begin{equation}
\frac PQ=\sum_{i=1}^m\frac1{x_i}+\frac{1}{T},
\label{eq:terminal-decomposition-identity}
\end{equation}
and
\begin{equation}
\begin{cases}
L\leq x_1\leq\cdots\leq x_m<T, & \quad\text{for nondecreasing denominators } (\sigma=\leq), \\
L<x_1<\cdots<x_m<T, & \quad\text{for strictly increasing denominators } (\sigma=<).
\end{cases}
\label{eq:terminal-decomposition-order}
\end{equation}
When $m=0$, either of the conditions means simply $L<T$. The decomposition stops at the unit fraction remainder $1/T$ and has value
\[
2^{-m}\mathcal{H}(T).
\]
Note that each control contributes the discount factor $1/2$ to this value.
The final term $1/T$ in \eqref{eq:terminal-decomposition-identity} is a remainder, not another selected control. When $m\geq1$, the strict condition $x_m<T$ ensures that the terminal state is admissible and can also be continued by its greedy tail of unit fractions. Every terminal decomposition \eqref{eq:terminal-decomposition-identity} determines a controlled trajectory:
\[ s=(P,Q,L), \quad s[x_1], \quad s[x_1][x_2], \quad \ldots, \quad s[x_1]\cdots[x_m]. \]
Before $x_j$ is selected, the still unselected terms in \eqref{eq:terminal-decomposition-identity} show that the current remainder is strictly larger than $1/x_j$; hence $x_j$ belongs to the relevant control set \eqref{eq:control-set}, while \eqref{eq:terminal-decomposition-order} supplies the cutoff condition.
The two Bellman functions (for nondecreasing and strictly increasing denominators) are therefore defined by
\[
\mathcal{B}_\sigma(P,Q,L):=\sup 2^{-m}\mathcal{H}(T),
\]
where the supremum ranges over all terminal decompositions from $(P,Q,L)$ under the denominator convention specified by $\sigma$.
Here we use the extended real numbers and interpret $\sup\emptyset:=-\infty$. The enlarged domain makes the Bellman recursion below closed under every legal control.

Define the \emph{stopping obstacle} by
\[
\mathcal{O}(P,Q,L)=
\begin{cases}
\mathcal{H}(Q/P), & \text{if } Q/P\in\mathbb{N}\text{ and } Q/P>L,\\
-\infty,&\text{otherwise}.
\end{cases}
\]
The following lemma records the so-called \emph{Bellman optimality principle}.

\begin{lemma}\label{lem:bellman-principle}
If $\sigma$ is either $\leq$ or $<$, then the function $\mathcal{B}_\sigma$ has the following properties.
\begin{enumerate}[
 label=\textup{(B\arabic*)},
 ref=\textup{(B\arabic*)}
]
\item\label{item:bellman-B1}
For every $(P,Q,L)\in\mathfrak{S}$ and $c\in\mathbb{N}$,
\[
\mathcal{B}_\sigma(cP,cQ,L)=\mathcal{B}_\sigma(P,Q,L).
\]

\item\label{item:bellman-B2}
The identity
\[
\mathcal{B}_\sigma(s)
=\max\left\{
\mathcal{O}(s),
\frac12\sup_{x\in\mathfrak{A}_\sigma(s)}
\mathcal{B}_\sigma(s[x])
\right\}
\]
holds for every $s\in\mathfrak{S}$.

\item\label{item:bellman-B3}
If $\mathcal{F}\colon\mathfrak{S}\to[-\infty,\infty]$ satisfies
\begin{equation}
\mathcal{F}(s)\geq\mathcal{O}(s)
\quad\text{and}\quad
\mathcal{F}(s)\geq\frac12\mathcal{F}(s[x])
\label{eq:bellman-supersolution}
\end{equation}
for every state $s\in\mathfrak{S}$ and every control
$x\in\mathfrak{A}_\sigma(s)$, then
\[
\mathcal{F}(s)\geq\mathcal{B}_\sigma(s)
\]
holds for every $s\in\mathfrak{S}$.
\end{enumerate}
\end{lemma}

\begin{proof}
The terminal decompositions and their values depend on $P$ and $Q$ only through the fraction $P/Q$, which proves \ref{item:bellman-B1}. A terminal decomposition is either trivial (and stops the process immediately) or has a first denominator $x$, after which its remaining terms form a terminal decomposition from $s[x]$. Splitting the supremum according to these two alternatives proves \ref{item:bellman-B2}.

To prove \ref{item:bellman-B3}, follow any terminal decomposition \eqref{eq:terminal-decomposition-identity} and define recursively
\[
s_0:=s,
\quad
s_j:=s_{j-1}[x_j]\quad\text{for } 1\leq j\leq m.
\]
Iterating the second inequality in \eqref{eq:bellman-supersolution} exactly $m$ times, and then using the obstacle inequality from \eqref{eq:bellman-supersolution} at the final state $s_m$, gives
\begin{align}
\mathcal{F}(s) & = \mathcal{F}(s_0) 
\geq 2^{-1}\mathcal{F}(s_1) \geq 2^{-2}\mathcal{F}(s_2) \geq\cdots \nonumber  \\
& \geq 2^{-m}\mathcal{F}(s_m) \geq 2^{-m}\mathcal{O}(s_m) = 2^{-m}\mathcal{H}(T). \label{eq:explicit-bellman-induction}
\end{align}
Taking the supremum over all terminal decompositions proves $\mathcal{F}(s)\geq\mathcal{B}_\sigma(s)$. Thus $\mathcal{B}_\sigma$ is precisely the least supersolution of \eqref{eq:bellman-supersolution}.
\end{proof}

Although the quotient $P/Q$ determines the remainder, we retain the unreduced integers $P$ and $Q$ because the cleared numerator $P$ will be the parameter of an induction later in the proof. The iteration \eqref{eq:explicit-bellman-induction} follows one fixed decomposition for finitely many steps. The argument below instead obtains uniform information about all decompositions by strong induction on $P$. 
%It first proves that the Bellman functions take finite values and their defining suprema are attained on admissible states with $P\leq Q$.

%%%%%

\subsection{States from unit and nonunit fractions}

We begin by recording the classical unit fraction extremality result in the form needed below. Besides settling the unit fraction case of the theorem, it will provide the boundary condition for the two Bellman functions and the induction basis for the attainment argument in the next subsection.

\begin{lemma}
\label{lem:unit-fraction-best}
Let $T$ be a positive integer, put $T_0:=T$, and define $(T_j)_{j\geq0}$ by the recurrence \eqref{eq:unit-tail-recurrence}.
Take $\sigma$ to be either $\leq$ or $<$. Then for every $m\geq1$, the greedy $m$-term denominator tuple for $1/T$ is the unique maximizing tuple, and
\begin{equation}
R_m^{\sigma}\Bigl(\frac{1}{T}\Bigr)
=\sum_{j=0}^{m-1}\frac1{T_j+1}
=\frac1T-\frac1{T_m}.
\label{eq:unit-fraction-best-formula}
\end{equation}
Consequently, the Bellman boundary condition
\begin{equation}
\mathcal{B}_\sigma(P,Q,L)=\mathcal{H}(Q/P)
\label{eq:bellman-boundary-condition}
\end{equation}
holds for every admissible state $(P,Q,L)$ for which $Q/P$ is an integer, and the immediately stopped decomposition (with no denominators) is optimal.
\end{lemma}

\begin{proof}
For $T=1$, uniqueness in both conventions is a classical result by Curtiss and Takenouchi \cite{Curtiss,Takenouchi}. For $T\geq2$, Nathanson's theorem for a reduced fraction $p/q$ with $p\mid q+1$ applies with $p=1$ and $q=T$, again giving uniqueness in both conventions \cite[Theorem~5]{Nathanson}. Summing the telescoping identity
\[
\frac1{T_j+1}=\frac1{T_j}-\frac1{T_{j+1}}
\]
in $j$ gives \eqref{eq:unit-fraction-best-formula}.

Now let $(P,Q,L)$ be admissible and put $t=Q/P\in\mathbb{N}$. The immediately stopped decomposition is terminal and has value $\mathcal{H}(t)$. If another terminal decomposition has $m\geq1$ selected denominators $x_1,\ldots,x_m$ and remainder $1/T$, then
\[
\frac1t-\frac1T=\sum_{i=1}^m\frac1{x_i}
\leq R_m^\sigma\Bigl(\frac{1}{t}\Bigr)
=\frac1t-\frac1{\iterPhi{m}(t)}.
\]
Hence $T\leq\iterPhi{m}(t)$, and properties \ref{item:payoff-H1} and \ref{item:payoff-H2} give
\[
2^{-m}\mathcal{H}(T)
\leq2^{-m}\mathcal{H}(\iterPhi{m}(t))
=\mathcal{H}(t).
\]
Taking the supremum over all allowed terminal decompositions of $P/Q$ proves \eqref{eq:bellman-boundary-condition} and the asserted optimality.
\end{proof}

We will repeatedly use the standard finite descent argument underlying the greedy method described by Fibonacci; see \cite{Fibonacci} and, for example, \cite{AmbroBarcau} for a modern formulation. Write a rational remainder, not necessarily in lowest terms, as $u=P/Q$ with integers $0<P\leq Q$, and put $t=Q/P$. If $t\notin\mathbb{N}$ and $x=G(u)=\lfloor t\rfloor+1$, then
\[
u-\frac1x=\frac{P'}{Q'},
\quad
P'=Px-Q=P(x-t),
\quad
Q'=Qx.
\]
Because $0<x-t<1$, the positive integer $P'$ is smaller than $P$. The numerator therefore decreases at every noninteger greedy step, so after finitely many such steps the remainder is a unit fraction. Moreover, for the reciprocal $t'=Q'/P'$ of the new remainder,
\begin{equation}
t'=\frac{Qx}{Px-Q}
=\frac{tx}{x-t}
=t+\frac{P}{P'}t^2
>t(t+1),
\quad
t'>x.
\label{eq:greedy-inverse-growth}
\end{equation}
Here the first strict inequality follows from $P/P'>1$, and the second from $0<x-t<1<t$. Thus the greedy denominators are strictly increasing as well.

This calculation gives the following \emph{drift property} of the payoff function.

\begin{lemma}
\label{lem:greedy-bellman-drift}
Let $0<u\leq1$ be rational, put $t=1/u$ and $x=G(u)$, and let
$t'=(u-1/x)^{-1}$. Then
\[
\frac12\mathcal{H}(t')\geq\mathcal{H}(t),
\]
with equality if and only if $t$ is an integer. Consequently, a greedy control $x$ only increases discounted payoff, even strictly until a unit fraction remainder is reached, and then preserves it along the subsequent unit fraction tail.
\end{lemma}

\begin{proof}
If $t\notin\mathbb{N}$, then \eqref{eq:greedy-inverse-growth}, strict monotonicity of $\mathcal{H}$, and \ref{item:payoff-H1} give
\[
\frac12\mathcal{H}(t')
>\frac12\mathcal{H}(\Phi(t))
=\mathcal{H}(t).
\]
If $t\in\mathbb{N}$, then $x=t+1$ and $t'=t(t+1)=\Phi(t)$, so \ref{item:payoff-H1} gives
equality.
\end{proof}

We next associate a positive integer numerator with every rational remainder
obtained from a finite underapproximation. 
Again let $u=P/Q$, $0<P\leq Q$, and $t=Q/P$.
For nondecreasing positive integers $x_1,\ldots,x_m$ whose reciprocal sum $\sum_j 1/x_j$ is smaller than $u$, define denominators $Q_j$ and cleared numerators $C_j$ by \eqref{eq:cleared-numerators}, as before.
Every $C_j$ is a positive integer and this quantity is simply the numerator obtained by ``clearing'' the denominator $Q_j$ from the current remainder. Our main comparison estimate concerns the first index at which $C_j$ becomes smaller than $P$.

To control the first fall below the initial numerator, we use the following standard product--sum comparison. Its special case is implicit in Soundararajan's proof \cite{Soundararajan} and appears explicitly in Kamio \cite[Lemma~6]{Kamio}; the same argument gives the additional generalization needed here.

\begin{lemma}
\label{lem:variable-comparison}
Let $u>0$, let
\[
a_1\geq\cdots\geq a_m>0,
\quad
b_1\geq\cdots\geq b_m>0,
\]
and let $\beta_1,\ldots,\beta_m>0$. Suppose that, for every $1\leq j\leq m$,
\begin{align}
\sum_{i=1}^j a_i+
\beta_j u\prod_{i=1}^j a_i&=u,
\label{eq:variable-reference}\\
\sum_{i=1}^j b_i+
\beta_j u\prod_{i=1}^j b_i&\leq u.
\label{eq:variable-competitor}
\end{align}
Then
\[
\sum_{i=1}^j b_i
\leq
\sum_{i=1}^j a_i\quad\text{for } 1\leq j\leq m.
\]
\end{lemma}

\begin{proof}
We argue by strong induction on $j$, starting with the empty sum. If $b_j\leq a_j$, the conclusion follows from the induction hypothesis at $j-1$. Suppose that $b_j>a_j$. If
\[
\prod_{i=1}^j b_i
\geq
\prod_{i=1}^j a_i,
\]
then \eqref{eq:variable-reference} and \eqref{eq:variable-competitor}
immediately give
\[
\sum_{i=1}^j b_i
\leq u-\beta_j u\prod_{i=1}^j b_i
\leq u-\beta_j u\prod_{i=1}^j a_i
=\sum_{i=1}^j a_i.
\]
It remains to treat the case in which the displayed product inequality is reversed.
Let $r$ be the largest index such that
\[
\prod_{i=r}^j b_i
<
\prod_{i=r}^j a_i.
\]
Then $r<j$, and maximality gives
\[
\prod_{i=r}^{\ell}b_i
<
\prod_{i=r}^{\ell}a_i
\quad\text{for } r\leq\ell\leq j;
\]
otherwise, division would produce the same strict inequality on a tail beginning after $r$. Hence the nonincreasing vectors $(\log a_r,\ldots,\log a_j)$ and $(\log b_r,\ldots,\log b_j)$ satisfy the hypotheses of the weak majorization form of Karamata's inequality \cite{Karamata}. Applied to the increasing convex function $s\mapsto e^s$, it yields
\begin{equation}\label{eq:the2sums}
\sum_{i=r}^j b_i
\leq
\sum_{i=r}^j a_i.
\end{equation}
Combining this with the induction hypothesis at $r-1$ completes the proof.

The last comparison can also be obtained from Muirhead's inequality as in \cite{Soundararajan}, or more directly as in \cite[Lemma~1.1]{AmbroBarcau}. Yet another proof of this last comparison can be given by raising the sums \eqref{eq:the2sums} to a large integer power and using the multinomial theorem; we omit the details.
\end{proof}

We now apply the preceding comparison principle at the first index where the cleared numerator falls below the initial numerator $P$. Let
\[
e_0:=u,
\quad
e_j:=\frac{e_{j-1}^2}{1+e_{j-1}}
\quad\text{for } j\geq1
\]
and
\[
\alpha_j:=\frac{e_{j-1}}{1+e_{j-1}}
= \frac{e_j}{e_{j-1}}
= e_{j-1} - e_j
\quad\text{for } j\geq1.
\]
Then $(\alpha_j)_{j\geq1}$ is decreasing and
\begin{equation}
e_j=u\prod_{i=1}^j\alpha_i,
\quad
\sum_{i=1}^j\alpha_i+e_j=u,
\quad
e_j=\frac1{\iterPhi{j}(t)}.
\label{eq:continuous-identities}
\end{equation}

\begin{lemma}
\label{lem:bellman-transition}
Let $P,Q$ be positive integers with $P\leq Q$, put $u=P/Q$ and $t=Q/P$, and let $m\geq1$. Suppose that $x_1\leq\cdots\leq x_m$ are positive integers whose reciprocal sum $\sum_i 1/x_i$ is smaller than $u$. Define $C_j$ by \eqref{eq:cleared-numerators}.

\begin{enumerate}[label=\textup{(\roman*)}]
\item If $C_j\geq P$ for every $1\leq j\leq m$, then
\begin{equation}
\left(
u-\sum_{i=1}^m\frac1{x_i}
\right)^{-1}
\leq\iterPhi{m}(t).
\label{eq:no-decrease-bound}
\end{equation}

\item Suppose that $m$ is the first index for which $C_m=d<P$; equivalently,
\begin{equation}
C_j\geq P,\quad 1\leq j<m,
\quad C_m=d<P.
\label{eq:first-decrease-hypothesis}
\end{equation}
Then
\begin{equation}
\left(
u-\sum_{i=1}^m\frac1{x_i}
\right)^{-1}
<\frac Pd\,\iterPhi{m}(t).
\label{eq:first-decrease-bound}
\end{equation}
If $t_m$ is the reciprocal of the remainder after the $m$th term, then
\begin{equation}\label{eq:first-decrease-order}
t_m > x_m.
\end{equation}
\end{enumerate}
\end{lemma}

\begin{proof}
Put $v_i=1/x_i$. Equation \eqref{eq:cleared-numerators} gives
\begin{equation}
u-\sum_{i=1}^jv_i
=\frac{C_j}{P}\,
u\prod_{i=1}^jv_i.
\label{eq:cleared-product-identity}
\end{equation}
If $C_j\geq P$ for every $j$, Lemma~\ref{lem:variable-comparison}, applied with $a_i=\alpha_i$, $b_i=v_i$, and $\beta_j=1$, gives
\[
\sum_{i=1}^m v_i
\leq
\sum_{i=1}^m\alpha_i.
\]
The actual remainder is therefore at least $e_m$, and \eqref{eq:no-decrease-bound} follows from \eqref{eq:continuous-identities}.

Now assume \eqref{eq:first-decrease-hypothesis} and set
\[
\delta:=\frac dP\in(0,1).
\]
Keep the first $m-1$ reference terms and replace $\alpha_m$ by
\[
\widetilde{\alpha}_m
=\frac{e_{m-1}}{1+\delta e_{m-1}}.
\]
The modified reference sequence is still decreasing. For $m=1$ this is immediate. If $m\geq2$, then $e_{m-2}\leq u\leq1$ and the inequality $\widetilde{\alpha}_m\leq\alpha_{m-1}$ is equivalent to $(1-\delta)e_{m-2}^2\leq1$, which is true as well.
For $j<m$, the identities in \eqref{eq:continuous-identities} remain unchanged, while at $j=m$ we have
\[
\sum_{i=1}^{m-1}\alpha_i+\widetilde{\alpha}_m
+\delta u
\Bigl(\prod_{i=1}^{m-1}\alpha_i\Bigr)
\widetilde{\alpha}_m
=u.
\]
On the competitor side, \eqref{eq:cleared-product-identity} gives the required inequalities with coefficient $1$ for $j<m$ and equality with coefficient $\delta$ for $j=m$. Lemma~\ref{lem:variable-comparison}, with $\beta_j=1$ for $j<m$ and $\beta_m=\delta$, yields
\[
\sum_{i=1}^m v_i
\leq
\sum_{i=1}^{m-1}\alpha_i+
\widetilde{\alpha}_m.
\]
Hence the final remainder is at least
\[
\frac{\delta e_{m-1}^2}{1+\delta e_{m-1}}
>\delta\frac{e_{m-1}^2}{1+e_{m-1}}
=\delta e_m.
\]
Taking reciprocals proves \eqref{eq:first-decrease-bound}.

Finally, let $Q_{m-1}$ be as in \eqref{eq:cleared-numerators}. We have $C_{m-1}\leq Q_{m-1}$, since the remainder before the last term is at most $1$.
From $d<P\leq C_{m-1}\leq Q_{m-1}$ and $t_m=Q_{m-1}x_m/d$, we obtain $t_m>x_m$, which is \eqref{eq:first-decrease-order}.
\end{proof}

We now express these transition estimates as a \emph{Bellman equation} on the part of the admissible region where $P\leq Q$. Fix $\sigma$ to be either $\leq$ or $<$, and let $s=(P,Q,L)$ be an admissible state with $P\leq Q$. A finite block of controls $\mathbf{x}=(x_1,\ldots,x_r)$ of length $r=r(\mathbf{x})\geq1$ is called a \emph{first descent block} from $s$ if $x_1,\ldots,x_m\geq2$,
\begin{align*}
L\leq x_1\leq\cdots\leq x_r
& \quad\text{for nondecreasing denominators } (\sigma=\leq), \\
L<x_1<\cdots<x_r
& \quad\text{for strictly increasing denominators } (\sigma=<),
\end{align*}
sum of its reciprocals $\sum_i 1/x_i$ is smaller than $P/Q$, and cleared numerators $C_j$ of all partial reciprocal sums $\sum_{i=1}^j 1/x_i$, given by \eqref{eq:cleared-numerators}, satisfy
\begin{equation}
C_j\geq P \quad \text{for } 1\leq j<r,
\quad
d:=C_r<P.
\label{eq:first-descent-block}
\end{equation}
Let $Q_r$ be the common denominator from \eqref{eq:cleared-numerators} and write
\[
s[\mathbf{x}] := s[x_1]\cdots[x_r] = (d,Q_r,x_r).
\]
Lemma~\ref{lem:bellman-transition}\,(ii) shows that this new state is admissible and has numerator $d<P$. Every admissible state with $P\leq Q$ has a terminal decomposition: one stops immediately if its reciprocal remainder is an integer, and otherwise its strictly increasing greedy continuation reaches a unit fraction. The following Bellman equation separates decompositions according to their first numerator descent; decompositions with no such descent are bounded by $\mathcal{H}(t)$.

\begin{lemma}
Let $s=(P,Q,L)$ be admissible with $P\leq Q$, and put $t=Q/P$.
If $t\in\mathbb{N}$, then the boundary condition
\eqref{eq:bellman-boundary-condition} holds. If $t\notin\mathbb{N}$, then
\begin{equation}
\mathcal{B}_\sigma(s)
=\sup_{\mathbf{x}}
2^{-r(\mathbf{x})}
\mathcal{B}_\sigma(s[\mathbf{x}]),
\label{eq:stopped-bellman-equation}
\end{equation}
where the supremum is over all first descent blocks $\mathbf{x}$ from $s$ under the denominator convention $\sigma$.
Moreover,
\begin{equation}
\mathcal{B}_\sigma(s)>\mathcal{H}(t)
\quad\text{when }t\notin\mathbb{N}.
\label{eq:bellman-strict-obstacle}
\end{equation}
\end{lemma}

\begin{proof}
The integer case is Lemma~\ref{lem:unit-fraction-best}. Suppose that $t\notin\mathbb{N}$. The first greedy control $x=\lfloor t\rfloor+1$ satisfies
\[
C_1=Px-Q\in\{1,2,\ldots,P-1\},
\]
so it is a first descent block of length one. The greedy numerator descent terminates at a unit fraction, and the greedy denominators are strictly increasing. Since $L<t<x$, this gives a terminal decomposition from $s$.
By Lemma~\ref{lem:greedy-bellman-drift}, its value $W$ satisfies
\begin{equation}
W>\mathcal{H}(t).
\label{eq:greedy-obstacle-gap}
\end{equation}
The suffix after the first greedy control is a terminal decomposition from the corresponding successor, so the right-hand side of \eqref{eq:stopped-bellman-equation} is at least $W$.

Concatenating a first descent block $\mathbf{x}$ with a terminal decomposition from $s[\mathbf{x}]$ gives a terminal decomposition from $s$ and multiplies its value by $2^{-r(\mathbf{x})}$. Hence the right-hand side of \eqref{eq:stopped-bellman-equation} is at most $\mathcal{B}_\sigma(s)$.

Conversely, consider an arbitrary terminal decomposition from $s$. If its cleared numerator never falls below $P$, Lemma~\ref{lem:bellman-transition}\,(i), monotonicity of $\mathcal{H}$, and \ref{item:payoff-H1} show that its value is at most $\mathcal{H}(t)$, and hence, by \eqref{eq:greedy-obstacle-gap}, smaller than the right-hand side of \eqref{eq:stopped-bellman-equation}. Otherwise, stop the decomposition at its first index $r$ with $C_r<P$. Its remaining part is a terminal decomposition from $s[\mathbf{x}]$, so its value is at most $2^{-r}\mathcal{B}_\sigma(s[\mathbf{x}])$. Taking the supremum over all terminal decompositions proves the reverse inequality in \eqref{eq:stopped-bellman-equation}. Since $\mathcal{B}_\sigma(s)\geq W>\mathcal{H}(t)$, we also obtain \eqref{eq:bellman-strict-obstacle}.
\end{proof}

%%%%%

\subsection{Finiteness and attainment of the Bellman functions}

We now verify the remaining Bellman properties on admissible states with $P\leq Q$: finiteness, attainment, and a growth estimate uniform in $Q$ and $L$ when the numerator $P$ is fixed. The proof is a strong induction on $P$. For a control trajectory, the induction hypothesis cannot necessarily be used after one step, since the new numerator may be at least $P$ and the new state may lie outside the admissible region.
Lemma~\ref{lem:bellman-transition} solves both problems at the first index at which the cleared numerator falls below the initial numerator $P$ of the block: the new numerator is smaller and \eqref{eq:first-decrease-order} places the new state back in the admissible region. We therefore group the initial string of controls ending at that first descent and then apply the induction hypothesis.

Define a sequence of nonnegative constants recursively by
\[
\Gamma_1:=0,
\quad
\Gamma_k := \frac12 \Bigl( \log (2k) + \max_{1\leq \ell<k}\Gamma_\ell \Bigr)\quad\text{for } k\geq2.
\]

\begin{lemma}
\label{lem:bellman-majorant}
If $\sigma$ is either $\leq$ or $<$, then for every admissible state $(P,Q,L)$ with $P\leq Q$,
\begin{equation}
\mathcal{H}(Q/P)\leq\mathcal{B}_\sigma(P,Q,L)
\leq\mathcal{H}(Q/P)+\Gamma_P<\infty.
\label{eq:bellman-envelope-bound}
\end{equation}
The first inequality is an equality when $P\mid Q$ and it is strict otherwise.
Moreover, if $\mathbf{x}$ is a first descent block of length $r$ with successor $(d,Q_r,x_r)$, then
\begin{equation}
2^{-r}\bigl(\mathcal{H}(Q_r/d)+\Gamma_d\bigr)
<\mathcal{H}(Q/P)+\Gamma_P.
\label{eq:concrete-main-inequality}
\end{equation}
\end{lemma}

\begin{proof}
Put $t=Q/P$. For a first descent block, write $t_r=Q_r/d$.
Lemma~\ref{lem:bellman-transition}\,(ii), strict monotonicity of $\mathcal{H}$, and \ref{item:payoff-H1}--\ref{item:payoff-H3} give
\begin{align}
2^{-r}\bigl(\mathcal{H}(t_r)+\Gamma_d\bigr)
&<2^{-r}\left(
\mathcal{H}\Bigl(\frac Pd\,\iterPhi{r}(t)\Bigr)+\Gamma_d
\right)\nonumber\\
&\leq2^{-r}\left(
\log\Bigl(\frac Pd\,\iterPhi{r}(t)\Bigr)+\log 2+\Gamma_d
\right)\nonumber\\
&\leq2^{-r}\left(
\mathcal{H}(\iterPhi{r}(t))+\log P+\log 2+\Gamma_d
\right)\nonumber\\
&\leq\mathcal{H}(t)+2^{-r}
\Bigl(\log P+\log 2+\max_{1\leq \ell<P}\Gamma_\ell\Bigr)\nonumber\\
&\leq\mathcal{H}(t)+\Gamma_P.
\label{eq:first-decrease-continuation-bound}
\end{align}
Here we used $d\geq1$, the lower logarithmic bound for
$\mathcal{H}(\iterPhi{r}(t))$, property~\ref{item:payoff-H1}, and $r\geq1$.
This proves \eqref{eq:concrete-main-inequality}.

We now perform the promised Bellman induction on $P$. When $P=1$, the reciprocal remainder $Q/P$ is an integer, so Lemma~\ref{lem:unit-fraction-best} gives $\mathcal{B}_\sigma(1,Q,L)=\mathcal{H}(Q)$.
Assume the upper bound in \eqref{eq:bellman-envelope-bound} for all smaller positive numerators. The same boundary condition settles the case $P\mid Q$.
If $P\nmid Q$, the stopped Bellman equation \eqref{eq:stopped-bellman-equation}, the induction hypothesis for every $d<P$, and \eqref{eq:concrete-main-inequality} yield
\[
\mathcal{B}_\sigma(P,Q,L)
\leq\mathcal{H}(Q/P)+\Gamma_P.
\]
This is the Bellman induction: a decomposition is stopped when its cleared numerator first falls below $P$, and \eqref{eq:concrete-main-inequality} returns the estimate to the initial state.

Finally, the lower bound and its equality statement follow from \eqref{eq:bellman-boundary-condition} and \eqref{eq:bellman-strict-obstacle}. This completes the induction and the proof.
\end{proof}

The next lemma proves existence of maximizing terminal decompositions.

\begin{lemma}
\label{lem:bellman-attainment}
If $\sigma$ is either $\leq$ or $<$, then the supremum defining $\mathcal{B}_\sigma(P,Q,L)$ is attained at every admissible state with $P\leq Q$.
\end{lemma}

\begin{proof}
We again use strong induction on $P$. The case $P=1$ and, more generally, every state with $P\mid Q$ is settled by the immediately stopped decomposition in \eqref{eq:bellman-boundary-condition}. Fix $P\geq2$, assume attainment for all smaller positive numerators, and consider an admissible state $s=(P,Q,L)$ with $P\nmid Q$. Put $t=Q/P$. Let $W$ be the value of the complete greedy terminal decomposition from $s$. Recall that it satisfies \eqref{eq:greedy-obstacle-gap}.

For any first descent block of length $r$ with successor $(d,Q_r,x_r)$, Lemma~\ref{lem:bellman-majorant} and the computation in \eqref{eq:first-decrease-continuation-bound} give
\begin{equation}
2^{-r}\mathcal{B}_\sigma(d,Q_r,x_r)
<\mathcal{H}(t)+2^{1-r}\Gamma_P.
\label{eq:first-descent-branch-bound}
\end{equation}
Choose an integer $M\geq2$ such that
\begin{equation}
2^{1-M}\Gamma_P<W-\mathcal{H}(t).
\label{eq:attainment-cutoff}
\end{equation}
No first descent block of length $r\geq M$ can therefore be optimal.

It remains to consider $1\leq r<M$ and $1\leq d<P$. We claim that for fixed $r$ and $d$ there are only finitely many first descent blocks $\mathbf{x}=(x_1,\ldots,x_r)$ from $s$ having $C_r=d$. Such a block satisfies
\begin{equation}
\frac PQ
=\sum_{i=1}^r\frac1{x_i}
+\frac d{Qx_1\cdots x_r}.
\label{eq:fixed-block-equation}
\end{equation}
Since $x_i\geq x_1$,
\[
\frac PQ
\leq\frac r{x_1}+\frac d{Qx_1^r}
\leq\frac{r+d/Q}{x_1},
\]
and hence
\[
x_1\leq\frac{rQ+d}{P}.
\]
Once $x_1,\ldots,x_{j-1}$ have been fixed, for $2\leq j\leq r$ put
\[
u_j := \frac PQ-\sum_{i=1}^{j-1}\frac1{x_i}>0,
\quad
c_j := \frac d{Qx_1\cdots x_{j-1}},
\quad
q_j := r-j+1.
\]
Equation \eqref{eq:fixed-block-equation} implies
\[
u_j
=\sum_{i=j}^r\frac1{x_i}+\frac{c_j}{x_j\cdots x_r}
\leq\frac{q_j}{x_j}+\frac{c_j}{x_j^{q_j}}
\leq\frac{q_j+c_j}{x_j},
\]
so $x_j\leq(q_j+c_j)/u_j$. Induction on $j$ proves the claim.

Let $\mathfrak{F}$ be the finite collection of all such blocks $\mathbf{x}$ with $1\leq r<M$ and $1\leq d<P$, retaining only those satisfying the appropriate order condition and \eqref{eq:first-descent-block}. For every $\mathbf{x}\in\mathfrak{F}$, its successor is admissible and has numerator $d<P$; the induction hypothesis therefore supplies a maximizing terminal decomposition from that successor. Hence every corresponding value
\[
2^{-r(\mathbf{x})}
\mathcal{B}_\sigma(s[\mathbf{x}])
\]
is attained. The first greedy control belongs to $\mathfrak{F}$, because its length is one and $C_1\in\{1,2,\ldots,P-1\}$, so $\mathfrak{F}$ is nonempty.
Let $V$ be the largest of these finitely many values. The complete greedy decomposition is available after its first control; consequently $V\geq W$.

By \eqref{eq:first-descent-branch-bound} and \eqref{eq:attainment-cutoff}, every block omitted because $r\geq M$ has value below $W\leq V$. The stopped Bellman equation \eqref{eq:stopped-bellman-equation} now gives $\mathcal{B}_\sigma(s)=V$. Concatenating the maximizing block with its inductively maximizing continuation produces a maximizing terminal decomposition from $s$. This completes the induction.
\end{proof}

It remains to treat initial remainders $P/Q$ larger than $1$. The next argument shows that only finitely many initial strings of denominators can be relevant: every sufficiently long string is suppressed by the discount factor before the remainder enters the region covered by the preceding induction.

\begin{lemma}
\label{lem:global-bellman-attainment}
If $\sigma$ is either $\leq$ or $<$, then $\mathcal{B}_\sigma(P,Q,L)$ is finite
and its defining supremum is attained at every admissible state $(P,Q,L)$.
\end{lemma}

\begin{proof}
The assertion for $P\leq Q$ follows from Lemmas~\ref{lem:bellman-majorant} and~\ref{lem:bellman-attainment}. Suppose that $P>Q$ and put $\lambda=P/Q>1$. Admissibility \eqref{eq:admissible-state} then forces $L=0$.

We first note that a terminal decomposition exists. Choose the least integer $k\geq2$ such that $\sum_{j=2}^k 1/j \geq\lambda$.
If equality holds, then the decomposition with denominators $2,3,\ldots,k-1$ and the remainder $1/k$ is a terminal decomposition in both conventions. Otherwise,
\[
0<u:=\lambda-\sum_{j=2}^{k-1}\frac{1}{j}<\frac{1}{k}.
\]
The usual rational greedy algorithm, starting from $u$ and stopping immediately if $u$ is a unit fraction, terminates after finitely many steps.
If it selects any denominators, the first is larger than $k$ and the selected denominators are strictly increasing. In all cases its terminal denominator is larger than $k$ and than every denominator it selects.
Appending this continuation to $2,3,\ldots,k-1$ again gives a terminal decomposition in both conventions. Fix one such decomposition and denote its positive value by $W$.

Consider any terminal decomposition from $(P,Q,0)$, with the denominators $x_1,\ldots,x_m$ and the remainder $1/T$, and write
\[
u_j:=\lambda-\sum_{i=1}^j\frac{1}{x_i},
\quad
Q_j:=Q\prod_{i=1}^j x_i,
\quad
P_j:=Q_j u_j
\quad\text{for } 0\leq j\leq m.
\]
Each $P_j$ is a positive integer. Here $u_0=\lambda$, $Q_0=Q$, and $P_0=P$. Since $\lambda>1$, we have $m\geq1$, and $u_m=1/T<1/x_m$.
Let $r$ be the first index for which $u_r<1/x_r$. For $j<r$ we have $u_j\geq1/x_j$, and hence
\[
u_{j-1}=u_j+\frac{1}{x_j}\leq2u_j.
\]
At the index $r$,
\[
u_{r-1}=u_r+\frac{1}{x_r}<\frac{2}{x_r}.
\]
It follows that
\begin{equation}
x_r<\frac{2^r}{\lambda}.
\label{eq:entrance-last-denominator}
\end{equation}
Indeed, iteration of the preceding inequalities gives $u_{r-1}\geq2^{1-r}\lambda$.

The state $s_r=(P_r,Q_r,x_r)$ is admissible and satisfies $P_r<Q_r$, because $0<u_r<1/x_r\leq1/2$.
Moreover, the denominators are ordered, so
\begin{equation}
Q_r=Q\prod_{i=1}^r x_i
\leq Qx_r^r<Q2^{r^2}.
\label{eq:entrance-cleared-denominator}
\end{equation}
Conversely, any ordered control block in the convention $\sigma$ for which $r$ is the first index with $u_r<1/x_r$ can be followed by an arbitrary terminal decomposition from $s_r$. Splitting terminal decompositions at this first index therefore gives
\begin{equation}
\mathcal{B}_\sigma(P,Q,0)
=\sup 2^{-r}\mathcal{B}_\sigma(P_r,Q_r,x_r),
\label{eq:entrance-bellman-equation}
\end{equation}
where the supremum is taken over all such initial blocks.

We will use the elementary estimate
\begin{equation}
\Gamma_k\leq\log(2k)\quad\text{for } k\geq1.
\label{eq:gamma-elementary-bound}
\end{equation}
This follows by induction from the definition of $\Gamma_k$.
Indeed, for $k\geq2$ the induction hypothesis gives
\[
\max_{1\leq \ell<k}\Gamma_\ell\leq\log(2(k-1))<\log(2k),
\]
and substitution into the defining recurrence for $(\Gamma_k)_{k\geq1}$ proves \eqref{eq:gamma-elementary-bound}. The Bellman upper bound \eqref{eq:bellman-envelope-bound}, the upper logarithmic bound for $\mathcal{H}$, and \eqref{eq:gamma-elementary-bound} now yield
\begin{align*}
\mathcal{B}_\sigma(P_r,Q_r,x_r)
&\leq\mathcal{H}(Q_r/P_r)+\Gamma_{P_r}\\
&\leq\log(Q_r/P_r)+\log 2+\log(2P_r)\\
&=\log Q_r+2\log 2.
\end{align*}
Consequently, by \eqref{eq:entrance-cleared-denominator},
\begin{equation}
2^{-r}\mathcal{B}_\sigma(P_r,Q_r,x_r)
<\eta_r:=2^{-r}\bigl(\log Q+(r^2+2)\log 2\bigr),
\label{eq:entrance-branch-bound}
\end{equation}
and $\eta_r\to0$ as $r\to\infty$.

Choose an integer $M$ so large that $\eta_r<W$ for every $r\geq M$. There are only finitely many possible initial blocks with $r<M$: by \eqref{eq:entrance-last-denominator}, each entry satisfies $x_i\leq x_r<2^r/\lambda<2^M$. Each successor belongs to the region $P_r<Q_r$, where its Bellman value is already known to be finite and attained. The initial block belonging to the fixed decomposition of value $W$ has length, say, $r_0<M$; otherwise \eqref{eq:entrance-branch-bound} would give
\[
W\leq2^{-r_0}\mathcal{B}_\sigma(P_{r_0},Q_{r_0},x_{r_0})
<\eta_{r_0}<W.
\]
Thus the finite collection is nonempty. Let $V$ be the largest of its branch values. Every branch with $r\geq M$ has value smaller than $W\leq V$, while every shorter branch has value at most $V$. Equation~\eqref{eq:entrance-bellman-equation} gives $\mathcal{B}_\sigma(P,Q,0)=V$. Concatenating a maximizing initial block with a maximizing continuation from its successor attains this value.
\end{proof}

%%%%%

\subsection{Completion of competing underapproximations}

We first show how to complete a finite underapproximation whose error is sufficiently small by a greedy tail. In the nondecreasing convention, the original tuple can be kept intact. In the strictly increasing convention, each collision with the greedy continuation is eliminated by deleting the colliding original denominator and restarting.

\begin{lemma}
\label{lem:completion}
Let $\lambda=p/q>0$ be in lowest terms, fix $\sigma$ to be either $\leq$ or $<$, and let $(a_1,\ldots,a_n)\in\mathcal{E}_n^\sigma$ satisfy $\sum_{i=1}^n 1/a_i <\lambda$.
Put
\[
h=\lambda-
\sum_{i=1}^n\frac1{a_i}.
\]
Assume that $h\leq1$ when $\sigma$ is $\leq$, and that $2^nh\leq1$ when $\sigma$ is $<$. Then there exist an integer $r$ with $0\leq r\leq n$ and a (possibly empty) subtuple $(a_{i_1},\ldots,a_{i_{n-r}})$ obtained by deleting $r$ occurrences from the original denominator tuple, with $i_1<\cdots<i_{n-r}$ when $r<n$, such that, on setting
\[
h'=\lambda-\sum_{j=1}^{n-r}\frac1{a_{i_j}},
\]
one has
\begin{equation}
h'\leq2^r h,
\label{eq:error-doubling-bound}
\end{equation}
and the Bellman value at the initial state satisfies
\begin{equation}
\mathcal{B}_\sigma(p,q,0)
\geq2^{-(n-r)}\mathcal{H}\left(\frac{1}{h'}\right).
\label{eq:completion-value}
\end{equation}
In the nondecreasing denominator convention ($\sigma=\leq$) one may take $r=0$ and $h'=h$.
\end{lemma}

\begin{proof}
First suppose that $\sigma$ is $\leq$. Starting from the rational remainder $h$, follow the greedy algorithm until a unit fraction remainder is reached. Continue along its unit fraction tail until the terminal denominator is larger than every original denominator. The reciprocal remainder after each noninteger step exceeds the denominator just selected by \eqref{eq:greedy-inverse-growth}, and the same is immediate from \eqref{eq:unit-tail-recurrence} along the unit fraction tail. Thus the final terminal denominator is also larger than every newly selected denominator. Repetitions are allowed, and after rearrangement the original denominators together with the greedy continuation form a nondecreasing terminal decomposition. Hence $r=0$ and $h'=h$.

Now suppose that $\sigma$ is $<$. Begin with the set
\[
A_0=\{a_1,\ldots,a_n\},
\]
and the error $h_0=h$. From $h_j$, generate the greedy continuation one denominator at a time, continuing along the unit fraction tail if a unit fraction remainder is reached. If a selected denominator first collides with $A_j$, denote it by $y$; discard the partial continuation, delete the original occurrence of $y$, and restart with
\[
A_{j+1}=A_j\setminus\{y\},
\quad
h_{j+1}=h_j+\frac1y.
\]
If there is no collision, stop once the terminal denominator is larger than every element of $A_j$. At the moment when $1/y$ is selected, it is strictly smaller than the current remainder, which is at most $h_j$. Therefore
\begin{equation}
h_{j+1}<2h_j.
\label{eq:single-collision-bound}
\end{equation}
Each restart deletes one original denominator, so the process terminates after at most $n$ restarts. After $j$ restarts, iteration of
\eqref{eq:single-collision-bound} gives
\[
h_j\leq2^jh\leq2^nh\leq1.
\]
Thus, every restarted greedy run lies in the range covered by Lemma~\ref{lem:greedy-bellman-drift}. If $r$ deletions have occurred, the same estimate gives \eqref{eq:error-doubling-bound}. In the final run, no selected denominator belongs to $A'=A_r$. The standard calculation leading to \eqref{eq:greedy-inverse-growth}, together with \eqref{eq:unit-tail-recurrence}, shows that the newly selected denominators are strictly increasing and all smaller than the final terminal denominator, which was also chosen larger than every element of $A'$. Rearrangement therefore gives a strictly increasing terminal decomposition from the initial state and the desired tuple is obtained from the set $A'$.

In either convention, suppose that the final greedy run uses $s$ selected denominators and ends with remainder $1/T$. Iterating the greedy Bellman drift from Lemma~\ref{lem:greedy-bellman-drift} over these $s$ controls gives
\[
2^{-s}\mathcal{H}(T)\geq \mathcal{H}\left(\frac{1}{h'}\right).
\]
The terminal decomposition contains $(n-r)+s$ selected denominators before the final unit fraction remainder, so its discounted payoff is at least $2^{-(n-r)}\mathcal{H}(1/h')$. This decomposition is among those in the definition of $\mathcal{B}_\sigma(p,q,0)$, which proves
\eqref{eq:completion-value}.
\end{proof}

With the completion lemma established, we can compare every sufficiently good competitor underapproximation with a maximizing terminal decomposition and deduce Theorem~\ref{thm:main} from the Bellman function properties.

\begin{proof}[Proof of Theorem~\ref{thm:main}]
Fix $\sigma$ to be either $\leq$ or $<$. Suppose first that $\lambda=1/T$ is a unit fraction. With $(T_j)_{j\geq0}$ defined by $T_0:=T$ and \eqref{eq:unit-tail-recurrence}, Lemma~\ref{lem:unit-fraction-best} gives
\[
R_{m+1}^{\sigma}\Bigl(\frac{1}{T}\Bigr)
=\sum_{j=0}^{m}\frac1{T_j+1}
=R_1^{\sigma}\Bigl(\frac{1}{T}\Bigr)
+R_m^{\sigma}\left(
\frac1T-R_1^{\sigma}\Bigl(\frac{1}{T}\Bigr)
\right)
\]
for every $m\geq1$.
Thus the theorem holds with $n_0=1$. The uniqueness and compatibility assertions follow from the same lemma.

Assume from now on that $\lambda$ is not a unit fraction, and write $\lambda=p/q$ in lowest terms. Apply Lemma~\ref{lem:global-bellman-attainment} to the initial state $(p,q,0)$ and choose a maximizing terminal decomposition. It cannot have length zero, since that would make $\lambda$ a unit fraction. Hence, there are integers $N\geq1$, $T\geq3$, and denominators $a_1,a_2,\ldots,a_N$ satisfying
\[
2\leq a_1\leq\cdots\leq a_N<T
\]
when $\sigma$ is $\leq$, and
\[
2\leq a_1<\cdots<a_N<T
\]
when $\sigma$ is $<$,
such that
\[
\lambda=
\sum_{i=1}^N\frac1{a_i}+\frac1T
\]
and
\[
\Lambda:=2^{-N}\mathcal{H}(T)=\mathcal{B}_\sigma(p,q,0).
\]
With $(T_j)_{j\geq0}$ defined by $T_0:=T$ and \eqref{eq:unit-tail-recurrence} again, for every $k\geq0$ we have
\[
\lambda=
\sum_{i=1}^N\frac1{a_i}
+\sum_{j=0}^{k-1}\frac1{T_j+1}
+\frac1{T_k}.
\]
The selected denominators form a tuple in $\mathcal{E}_{N+k}^{\sigma}$ whose reciprocal sum is $\lambda-1/T_k$. By \ref{item:payoff-H1},
\[
2^{-(N+k)}\mathcal{H}(T_k)=\Lambda.
\]

Consider first the nondecreasing denominator convention. Fix
$k\geq0$ and put $n=N+k$. Suppose, for contradiction, that there
is an $n$-term underapproximation larger than
$\lambda-1/T_k$. Its error $h$ then satisfies
\begin{equation}
h<\frac1{T_k}.
\label{eq:better-error}
\end{equation}
The Bellman completion estimate \eqref{eq:completion-value}, with $r=0$ and $h'=h$, applies because $h<1/T_k\leq1/T<1$ and it would give
\[
\Lambda=\mathcal{B}_{\leq}(p,q,0)
\geq2^{-n}\mathcal{H}\Bigl(\frac{1}{h}\Bigr)
>2^{-n}\mathcal{H}(T_k)
=\Lambda,
\]
a contradiction. Hence
\begin{equation}
R_{N+k}^{\leq}(\lambda)
=\lambda-\frac1{T_k}\quad\text{for } k\geq0.
\label{eq:nondecreasing-best-extension}
\end{equation}

Now consider the strictly increasing convention. Choose a nonnegative integer $k_0$ so large that, whenever $k\geq k_0$ and $n=N+k$,
\begin{equation}
T_k>8,
\quad
\log T_k>2n\log 2.
\label{eq:tail-growth-choice}
\end{equation}
Such a choice is possible because $T\geq3$, so $\mathcal{H}(T)\geq\log T>0$, and
\[
\mathcal{H}(T_k)=2^k\mathcal{H}(T),
\quad
\log T_k\geq \mathcal{H}(T_k)-\log 2.
\]
Thus $\log T_k\geq2^k\mathcal{H}(T)-\log 2$, which eventually dominates $2(N+k)\log 2$. Fix $k\geq k_0$ and put $n=N+k$. Suppose, for contradiction, that there is a strictly increasing $n$-term underapproximation larger than $\lambda-1/T_k$. Its error again satisfies
\eqref{eq:better-error}. Moreover, \eqref{eq:tail-growth-choice} gives
\[
2^nh<\frac{2^n}{T_k}<1.
\]
Lemma~\ref{lem:completion} therefore gives an integer $0\leq r\leq n$ and a modified error $h'$ satisfying \eqref{eq:error-doubling-bound}. Since $r\leq n$, \eqref{eq:tail-growth-choice} gives
\[
\log\Bigl(\frac{T_k}{2^r}\Bigr)
=\log T_k-r\log 2
>\frac12\log T_k>0,
\]
so $T_k/2^r>1$.
If $r=0$, then $2^r\mathcal{H}(T_k/2^r)=\mathcal{H}(T_k)$. If $r=1$, the bounds in \ref{item:payoff-H3} and $T_k>8$ give
\[
2\mathcal{H}(T_k/2)
\geq2\log(T_k/2)
>\log T_k+\log 2
\geq\mathcal{H}(T_k).
\]
If $r\geq2$, then
\[
2^r\mathcal{H}(T_k/2^r)
\geq4\log(T_k/2^r)
>2\log T_k
>\log T_k+\log 2
\geq\mathcal{H}(T_k).
\]
Thus, in every case,
\[
2^r\mathcal{H}(T_k/2^r)\geq\mathcal{H}(T_k).
\]
On the other hand, \eqref{eq:completion-value}, \eqref{eq:error-doubling-bound}, \eqref{eq:better-error}, and strict monotonicity of $\mathcal{H}$ give a contradiction:
\[
\Lambda
\geq2^{-(n-r)}\mathcal{H}\Bigl(\frac{1}{h'}\Bigr)
>2^{-n}2^r\mathcal{H}\Bigl(\frac{T_k}{2^r}\Bigr)
\geq2^{-n}\mathcal{H}(T_k)
=\Lambda.
\]
Therefore
\begin{equation}
R_{N+k}^{<}(\lambda)
=\lambda-\frac1{T_k}\quad\text{for } k\geq k_0.
\label{eq:strict-best-extension}
\end{equation}

Under either denominator convention, define
\[
n_0 :=
\begin{cases}
N & \text{if } \sigma=\leq,\\
N+k_0 & \text{if } \sigma=<
\end{cases}
\]
and
\[
U_j :=
\begin{cases}
T_j & \text{if } \sigma=\leq,\\
T_{k_0+j} & \text{if } \sigma=<
\end{cases}
\]
for every $j\geq0$.
Equations \eqref{eq:nondecreasing-best-extension} and \eqref{eq:strict-best-extension}, in their respective cases, give
\[
R_{n_0}^{\sigma}(\lambda)
=\lambda-\frac1{U_0}.
\]
For every $m\geq1$,
\begin{equation}\label{eq:final-extension-decomposition}
R_{n_0+m}^{\sigma}(\lambda)
=\lambda-\frac1{U_m}
=R_{n_0}^{\sigma}(\lambda)
+\sum_{j=0}^{m-1}\frac1{U_j+1}.
\end{equation}
Lemma~\ref{lem:unit-fraction-best} identifies the sum in \eqref{eq:final-extension-decomposition} with
$R_m^{\sigma}(1/U_0)$.
Since
\[
\frac1{U_0}
=\lambda-R_{n_0}^{\sigma}(\lambda),
\]
equation \eqref{eq:main-decomposition} follows, and Lemma~\ref{lem:unit-fraction-best} shows that the tuple formed by the final $m$ greedy denominators is the unique maximizing tuple for the remaining unit fraction. The prefixes of the constructed sequence at lengths $n_0$ and $n_0+m$ supply the compatible maximizing tuples asserted in the theorem.
This proves Theorem~\ref{thm:main}.
\end{proof}

%%%%%

\section{Proof of Theorem~\ref{thm:asym}}
\label{sec:asymptotic-rigidity}

The preceding Bellman argument also gives the asserted strengthening of Corollary~\ref{cor:li-tang}. Let $\lambda=p/q>0$ be in lowest terms and put $\Lambda:=\mathcal{B}_{\leq}(p,q,0)$.
If a terminal decomposition with $m$ selected denominators and terminal remainder $1/T$ is extended by the greedy unit fraction tail into $\lambda=\sum_n 1/c_n$, then Lemma~\ref{lem:payoff-construction} gives
\begin{equation}\label{eq:terminal-ray-growth}
\log\left(\lim_{n\to\infty}c_n^{2^{-n}}\right)
=2^{-m-1}\mathcal{H}(T).
\end{equation}
Indeed,
\[
c_{m+j+1}=\iterPhi{j}(T)+1\quad\text{for }j\geq0.
\]
Since the sequence $(b_n)_{n\geq1}$ comes from a maximizing terminal decomposition, it follows that
\begin{equation}\label{eq:maximal-ray-growth}
\log\left(\lim_{n\to\infty}b_n^{2^{-n}}\right)
=\frac12\Lambda.
\end{equation}
Every nondecreasing expansion of $\lambda$ with an eventual greedy unit fraction tail determines a terminal decomposition by cutting sufficiently far down that tail. Consequently, \eqref{eq:terminal-ray-growth} and the definition of the Bellman function show that the corresponding limit of $c_n^{2^{-n}}$ is at most $\lim_{n\to\infty}b_n^{2^{-n}}$, with equality precisely when the terminal decomposition is maximizing.

We need one further consequence of the attainment argument. Throughout this paragraph we consider only the nondecreasing denominator convention. Call an infinite denominator sequence from an admissible state $s$ a \emph{maximizing ray} if it is obtained by extending a maximizing terminal decomposition from $s$ by the greedy unit fraction tail of its terminal remainder. Different stopping points producing the same infinite sequence are regarded as belonging to the same ray. We claim that every admissible state has only finitely many maximizing rays.

We first prove this claim for all admissible states $s=(P,Q,L)$ with $P\leq Q$, by strong induction on $P$. Suppose first that $T:=Q/P\in\mathbb{N}$.
Then \eqref{eq:bellman-boundary-condition} gives
\[
\mathcal{B}_{\leq}(s)=\mathcal{H}(T).
\]
Consider a maximizing terminal decomposition of length $r$ with terminal remainder $1/T'$. Its optimality gives
\[
2^{-r}\mathcal{H}(T')=\mathcal{H}(T),
\]
and hence property~\ref{item:payoff-H1} yields
\[
\mathcal{H}(T')
=2^r\mathcal{H}(T)
=\mathcal{H}\bigl(\iterPhi{r}(T)\bigr).
\]
Strict monotonicity of $\mathcal{H}$ therefore implies
\[
T'=\iterPhi{r}(T).
\]
If $r\geq1$, the selected denominators consequently attain
\[
\frac1T-\frac1{\iterPhi{r}(T)}
=R_r^{\leq}\Bigl(\frac1T\Bigr),
\]
so Lemma~\ref{lem:unit-fraction-best} forces them to be the first $r$ greedy denominators, while the case $r=0$ is immediate. Thus every maximizing stopping point produces the same infinite greedy ray, and there is precisely one maximizing ray from $s$.

Now suppose that $Q/P\notin\mathbb{N}$, and assume the claim for every admissible state whose cleared numerator is smaller than $P$. Every maximizing terminal decomposition from $s$ must have a first descent block: otherwise Lemma~\ref{lem:bellman-transition}\,(i) would bound its value by $\mathcal{H}(Q/P)$, contradicting \eqref{eq:bellman-strict-obstacle}. Choose $W$ and $M$ as in the proof of Lemma~\ref{lem:bellman-attainment}. The estimates \eqref{eq:first-descent-branch-bound} and \eqref{eq:attainment-cutoff} show that the first descent block of a maximizing decomposition has length $r<M$. For every $r<M$ and every $1\leq d<P$, the argument based on \eqref{eq:fixed-block-equation} gives only finitely many first descent blocks of length $r$ whose final cleared numerator is $d$. Hence only finitely many first descent blocks can begin a maximizing terminal decomposition.

Let $\mathbf{x}$ be one such block, put $s':=s[\mathbf{x}]$, and let $V'$ be the value of the terminal suffix following $\mathbf{x}$. If the whole terminal decomposition is maximizing, then \eqref{eq:stopped-bellman-equation} gives
\[
\mathcal{B}_{\leq}(s)
=2^{-r(\mathbf{x})}V'
\leq
2^{-r(\mathbf{x})}\mathcal{B}_{\leq}(s')
\leq
\mathcal{B}_{\leq}(s).
\]
Equality therefore holds throughout, so the suffix is itself a maximizing terminal decomposition from $s'$. The state $s'$ is admissible and has cleared numerator $d<P$. By the induction hypothesis, only finitely many maximizing rays continue from $s'$. Since there are only finitely many possible initial blocks $\mathbf{x}$, there are only finitely many maximizing rays from $s$. This completes the strong induction.

Finally, suppose that $P>Q$. Admissibility then forces $L=0$. Choose $W$ and the cutoff $M$ as in the proof of Lemma~\ref{lem:global-bellman-attainment}. The estimate \eqref{eq:entrance-branch-bound} shows that the entrance block of any maximizing terminal decomposition has length $r<M$. Moreover, \eqref{eq:entrance-last-denominator} gives
\[
x_i\leq x_r
<\frac{2^r}{P/Q}
<2^M,
\]
so there are only finitely many possible entrance blocks. If $s_r$ is the successor after such a block, the same equality argument, now using \eqref{eq:entrance-bellman-equation}, shows that the remaining suffix must be maximizing from $s_r$. Since $s_r$ is admissible and satisfies $P_r<Q_r$, the preceding part of the argument shows that only finitely many maximizing rays continue from $s_r$. It follows that there are only finitely many maximizing rays from $(P,Q,0)$. In particular, the initial state $(p,q,0)$ has only finitely many maximizing rays.

Moreover, all of these maximizing infinite sequences are eventually equal. Indeed, suppose that two maximizing terminal decompositions have lengths $m\leq m'$ and terminal denominators $T$ and $T'$. Equality of their values, property~\ref{item:payoff-H1}, and strict monotonicity of $\mathcal{H}$ give
\[
\mathcal{H}(T')
=2^{m'-m}\mathcal{H}(T)
=\mathcal{H}\bigl(\iterPhi{(m'-m)}(T)\bigr),
\]
i.e.,
\[
T'=\iterPhi{(m'-m)}(T).
\]
Their greedy tails therefore agree from index $m'+1$ onward. In particular, every maximizing infinite sequence is eventually equal to $(b_n)_{n\geq1}$.

Now let $(a_n)_{n\geq1}$ be as in Theorem~\ref{thm:asym}. For each $k\geq1$, the rational tail
\[
\mu_k:=\lambda-\sum_{i=1}^k\frac1{a_i}
\]
is positive. Apply the Bellman construction in the nondecreasing case of Theorem~\ref{thm:main} to $\mu_k$, and denote the resulting compatible eventually greedy expansion by $(d_j^{(k)})_{j\geq1}$. Arrange the multiset union of $a_1,\ldots,a_k$ and all the $d_j^{(k)}$ in nondecreasing order, obtaining a sequence $(c_n^{(k)})_{n\geq1}$. Since $d_j^{(k)}\to\infty$, one has $c_{k+j}^{(k)}=d_j^{(k)}$ for all sufficiently large $j$. Thus $(c_n^{(k)})_{n\geq1}$ has an eventual greedy unit fraction tail, has reciprocal sum $\lambda$, and contains $a_1,\ldots,a_k$ with their multiplicities. Cutting sufficiently far down that tail, at an index $M$ for which $c_M^{(k)}-1>c_{M-1}^{(k)}$, gives a legal terminal decomposition with terminal denominator $c_M^{(k)}-1$.

The following truncation comparison is adapted from the proof of \cite[Theorem~1.12]{LiTangSylvester}:
\begin{equation}\label{eq:truncation-root-comparison}
\liminf_{n\to\infty}a_n^{2^{-n}}
\leq
\lim_{n\to\infty}\bigl(c_n^{(k)}\bigr)^{2^{-n}}.
\end{equation}
For completeness, write $e_j:=a_{k+j}$ and $D_k:=\lim_{j\to\infty}(d_j^{(k)})^{2^{-j}}$. If $\liminf_{j\to\infty}e_j^{2^{-j}}>D_k$, then $e_j>d_j^{(k)}$ for all sufficiently large $j$. On the other hand, every sufficiently long prefix of $(d_j^{(k)})_{j\geq1}$ is a best underapproximation of $\mu_k$, and hence
\[
\sum_{j=1}^N\frac1{e_j}
\leq
\sum_{j=1}^N\frac1{d_j^{(k)}}
\]
for all sufficiently large $N$. The assumed eventual inequality $e_j>d_j^{(k)}$ also gives
\[
\sum_{j>N}\frac1{e_j}
<
\sum_{j>N}\frac1{d_j^{(k)}}
\]
for all sufficiently large $N$, and adding this to the preceding prefix inequality would contradict the fact that both reciprocal series sum to $\mu_k$. Thus
\[
\liminf_{j\to\infty}e_j^{2^{-j}}\leq D_k.
\]
Moreover,
\[
\liminf_{j\to\infty}e_j^{2^{-j}}
=\left(\liminf_{n\to\infty}a_n^{2^{-n}}\right)^{2^k},
\quad
\lim_{n\to\infty}\bigl(c_n^{(k)}\bigr)^{2^{-n}}
=D_k^{2^{-k}},
\]
where the second identity follows because inserting the $k$ fixed denominators shifts the eventual sequence $(d_j^{(k)})_{j\geq1}$ by exactly $k$ places. This proves \eqref{eq:truncation-root-comparison}.

By \eqref{eq:terminal-ray-growth}, \eqref{eq:maximal-ray-growth}, and \eqref{eq:truncation-root-comparison},
\[
\liminf_{n\to\infty}a_n^{2^{-n}} \leq \lim_{n\to\infty}b_n^{2^{-n}}
\]
for every $k$. If the strict inequality in Theorem~\ref{thm:asym} fails, then equality holds throughout, so every $(c_n^{(k)})_{n\geq1}$ is one of the finitely many maximizing infinite sequences described above. One fixed maximizing sequence must therefore occur for arbitrarily large $k$. It contains every $a_n$, with multiplicity; since its reciprocal sum is also $\lambda$, it contains no additional terms. The two nondecreasing sequences are identical, and the preceding argument shows that $a_n=b_n$ for all sufficiently large $n$, as required.
This proves Theorem~\ref{thm:asym}.

Under the uniqueness hypothesis of Corollary~\ref{cor:li-tang}, the eventual equality alternative forces equality of the entire sequences: compare a sufficiently long common tail and then use uniqueness of the corresponding maximizing prefix. Thus Theorem~\ref{thm:asym} also recovers Corollary~\ref{cor:li-tang}.

%%%%%

\section{Elaboration of Example~\ref{ex:Liouville}}
\label{sec:example}

For $n\geq1$ put
\[
b_n:=2^{n!}
\]
and then for $n\geq0$ define
\[
S_n:=\sum_{j=1}^n\frac1{b_j},\quad
\tau_n:=\theta-S_n.
\]
We can write $S_n$ as a single fraction, $S_n=P_n/b_n$, with numerator
\[
P_n:=\sum_{j=1}^n2^{n!-j!}.
\]
The integer $P_n$ is odd because its final summand is $1$ and all preceding terms are even. Thus $b_n$ is the reduced denominator of $S_n$.

We will also use the separation estimate
\begin{equation}\label{eq:liouville-tail-separation}
\tau_n
<\frac1{(b_n-1)b_n^n}
\leq\frac1{b_n(b_n-1)^n}
\quad\text{for } n\geq1.
\end{equation}
Indeed, $k!\geq n!k$ for $k\geq n+1$, with strict inequality when $k\geq n+2$, and therefore
\[
\tau_n
<\sum_{k=n+1}^{\infty}\frac1{b_n^k}
=\frac{1/b_n^{n+1}}{1-1/b_n}
=\frac1{(b_n-1)b_n^n}.
\]
The same estimate gives $0<\theta-S_n<b_n^{-n}$, so for every positive integer $m$, taking $n>m$ gives $0<\theta-S_n<b_n^{-m}$. Since $S_n=P_n/b_n$ is reduced, $\theta$ is indeed a Liouville number.

We now prove by induction the stronger assertion that every nondecreasing tuple
\[
2\leq a_1\leq\cdots\leq a_n,
\quad
A:=\sum_{j=1}^n\frac1{a_j}<\theta,
\]
satisfies $A\leq S_n$, with equality if and only if $(a_1,\ldots,a_n)=(b_1,\ldots,b_n)$. In particular,
\begin{equation}\label{eq:liouville-best-values}
R_n^{\leq}(\theta)=S_n\quad\text{for } n\geq1.
\end{equation}

For $n=1$, we only need to observe $1/2<\theta<1$, so $1/a_1\leq1/2=S_1$, with equality if and only if $a_1=2=b_1$.

Let $n\geq2$, assume the assertion for $n-1$ terms, and put $A':=\sum_{j=1}^{n-1}1/a_j$. Since $A'<\theta$, the induction hypothesis gives $A'\leq S_{n-1}$.
If $A>S_n$, then
\[
\frac1{a_n}=A-A'
>S_n-S_{n-1}
=\frac1{b_n}.
\]
Consequently,
\[
a_j\leq b_n-1\quad\text{for } 1\leq j\leq n.
\]
Since $b_j\mid b_n$ for $j\leq n$, the positive number $A-S_n$, when written over the common denominator $b_n\prod_{j=1}^n a_j$, has a positive integral numerator. Hence, by \eqref{eq:liouville-tail-separation},
\[
A-S_n
\geq\frac1{b_n\prod_{j=1}^n a_j}
\geq\frac1{b_n(b_n-1)^n}
>\tau_n.
\]
This contradicts $A<\theta=S_n+\tau_n$, and therefore $A\leq S_n$.

Suppose now that $A=S_n$. The inequality $A'\leq S_{n-1}$ now yields
\[
\frac1{a_n}=S_n-A'
\geq S_n-S_{n-1}
=\frac1{b_n},
\]
so $a_n\leq b_n$. Since $b_n$ is the reduced denominator of $S_n$, the identity
\[
S_n=\sum_{j=1}^n\frac1{a_j}
\]
implies that $b_n$ divides the least common multiple of $a_1,\ldots,a_n$.
As $b_n$ is a power of $2$, it divides at least one of the denominators $a_j$. But $a_j\leq a_n\leq b_n$, so $a_j=b_n$ and, by monotonicity, $a_n=b_n$. Removing the final term gives $A'=S_{n-1}$, and the induction hypothesis shows that
\[
(a_1,\ldots,a_{n-1})=(b_1,\ldots,b_{n-1}).
\]
This completes the induction and proves both \eqref{eq:liouville-best-values} and uniqueness in the nondecreasing denominator convention. Since $(b_1,\ldots,b_n)$ is strictly increasing, it is also the unique maximizing tuple in the strictly increasing convention, and $R_n^{<}(\theta)=S_n$ for every $n\geq1$.

It remains to identify the greedy choices. Immediately before $1/b_n$ is selected, the remainder is
\[
\theta-S_{n-1}=\frac1{b_n}+\tau_n.
\]
Estimate \eqref{eq:liouville-tail-separation} yields
\[
\frac1{b_n}
< \theta-S_{n-1}
< \frac1{b_n}+\frac1{b_n(b_n-1)}
= \frac1{b_n-1}.
\]
Thus $G(\theta-S_{n-1})=b_n$ for every $n\geq1$, where $G$ was defined in \eqref{eq:def-of-G}. Hence the unique maximizing tuples are generated greedily from the first term onward.

%%%%%

\section*{Declaration of AI usage}
Key steps in the proof of Theorem~\ref{thm:main} (the choice and the construction of the payoff function, the existence of optimal terminal decompositions, and the study of competing nongreedy underapproximations) were provided by OpenAI's GPT-5.6 Sol. The authors placed them in the natural context of optimal control theory and rewrote them accordingly, using the ChatGPT-generated outputs as a starting point. Figure~\ref{fig:graph-of-H} was also created with GPT-5.6 Sol. The final statements of the results, the complete proofs presented here, and the remaining manuscript text were written by the authors, who take full responsibility for correctness and originality of the paper.

For transparency, we also describe more precisely how ChatGPT was used. We posed the problem independently through two separate ChatGPT accounts on July 13, 2026. In one conversation, ChatGPT reported that it was unable to solve the problem; the corresponding transcript is publicly available as the
\href{https://github.com/QuanyuTang/eventually-greedy-egyptian-candidate-proof/blob/main/02_GPT_proof_attempt_unsuccessful.pdf}
{\texttt{unsuccessful attempt}}.
In the other conversation, ChatGPT produced a candidate approach that appeared capable of resolving the problem; that transcript is available as the
\href{https://github.com/QuanyuTang/eventually-greedy-egyptian-candidate-proof/blob/main/01_GPT_proof_attempt_successful.pdf}
{\texttt{successful attempt}}.
We subsequently asked ChatGPT to organize the reasoning developed in the latter conversation into a preliminary proof draft. This GPT-generated draft remains publicly available as the
\href{https://github.com/QuanyuTang/eventually-greedy-egyptian-candidate-proof/blob/main/main2.pdf}
{\texttt{initial proof draft}}.

We emphasize that this initial draft was far from being a self-contained and comprehensible paper. At best, it was an essentially correct but undigested and barely readable proof. Considerable mathematical verification, restructuring, exposition, and rewriting by the authors were required to produce the present manuscript.

%%%%%

\section*{Acknowledgments and funding}
Q.\,T. would like to thank Zheng Li, Yinchen Liu, and Shengtong Zhang for helpful discussions.
Both authors are indebted to Wouter van Doorn for suggesting the inclusion of Theorem~\ref{thm:asym} and to Xiao Hu for Remark~\ref{rem:Hinterpretation}.
The authors are also grateful to Thomas Bloom for creating and maintaining the website \emph{Erd\H{o}s Problems} \cite{Bloom}.

V.\,K. is supported in part by the Croatian Science Foundation under the project HRZZ-IP-2022-10-5116 (FANAP) and in part by the European Union -- NextGenerationEU through the National Recovery and Resilience Plan 2021--2026, via an institutional grant of the University of Zagreb Faculty of Science, IK IA 1.1.3, Impact4Math.

%%%%%

\bibliographystyle{plainurl}
\bibliography{underapproximations}

\end{document}